\documentclass[reqno,12pt]{amsart}

\usepackage{color}
\usepackage{enumerate}
\usepackage[initials]{amsrefs}

\numberwithin{equation}{section}

\theoremstyle{plain}
\newtheorem{thm}{Theorem}[section]
\newtheorem{lemma}[thm]{Lemma} 
\newtheorem{prop}[thm]{Proposition}

\newtheorem{hyp}[thm]{Hypotheses}

\theoremstyle{remark}

\newtheorem{remark}[thm]{Remark}

\theoremstyle{definition}

\newtheorem{defi}[thm]{Definition}
\newtheorem{example}[thm]{Example}

\usepackage{amscd,amssymb,comment,epic,eepic,euscript,graphics}
\newcommand\eps{\epsilon}
\newcommand{\la}{\lambda}

\newcommand\Ic{{\mathcal{I}}}
\newcommand\Bc{{\mathcal{B}}}
\newcommand\Fc{{\mathcal{F}}}
\newcommand\Hc{{\mathcal{H}}}
\newcommand\Kc{{\mathcal{K}}}
\newcommand\Lc{{\mathcal{L}}}
\newcommand\Nc{{\mathcal{N}}}

\newcommand\Reals{{\mathbb R}}
\newcommand\Complex{{\mathbb C}}
\newcommand\Nats{{\mathbb N}}
\newcommand\Z{{\mathbb Z}}

\newcommand\BH{B(\Hc)}
\newcommand\BK{B(\Kc)}
\newcommand{\Hcal}{\mathcal{H}}
\newcommand\Tr{{\mathrm{Tr}}}

\newcommand{\im}{\text{\rm Im}}
\newcommand{\re}{\text{\rm Re}}
\newcommand{\ii}{\text{\rm i}}

\newcommand{\sign}{\text{\rm sign}}

\newcommand\mut{{\tilde{\mu}}}
\newcommand\psit{{\tilde{\psi}}}
\newcommand\diag{\text{\rm diag}}

\newcommand\restrict{{\upharpoonright}}

\newcommand\Ict{{\widetilde\Ic}}

\begin{document}

\title{Perturbation formulas for traces on normed ideals}

\author[Dykema]{Ken Dykema$^{1}$}
\address{K.D., Department of Mathematics, Texas A\&M University,
College Station, TX 77843-3368, USA}
\email{kdykema@math.tamu.edu}

\author[Skripka]{Anna Skripka$^2$}
\address{A.S., Department of Mathematics and Statistics,
	University of New Mexico, 400 Yale Blvd NE, MSC01 1115, Albuquerque, NM 87131, USA}
\email{skripka@math.unm.edu}

\thanks{\footnotesize ${}^1$Research supported in part by NSF grant DMS--1202660.
${}^2$Research supported in part by NSF grant DMS--1249186.}



\begin{abstract}
We prove perturbation results for traces on normed ideals in semifinite von Neumann algebra factors.
This includes the case of Dixmier traces.
In particular, we establish existence of spectral shift measures
with initial operators being dissipative or bounded,
and show that these measures can have singular components in the case of Dixmier traces.
We also establish a linearization formula for a Dixmier trace applied to perturbed operator functions, a result that does not typically hold for normal traces.
\end{abstract}

\maketitle

\section{Introduction}


The goal of this paper is to extend important results of perturbation theory for ideals with normal traces to more general operator ideals, including Marcinkiewicz ideals endowed with Dixmier traces, and obtain new results that are distinctive of singular traces. In particular, we establish existence of spectral shift measures, which are not always absolutely continuous. 
We recall that the spectral shift measures originate from research in physics \cite{Lifshits} (see also~\cite{Y}
and \cite{RPKF}); they have been applied in perturbation theory of Schr\"{o}dinger operators and in noncommutative geometry in the study of spectral flow \cite{sf}. Existence of absolutely continuous spectral shift measures linked to normal traces was established in \cites{Krein,Kreinunitary,CP,AzamovIE} and of second order spectral shift measures in \cites{Koplienko,Neidhardt,DS,PS}.
Singular traces are important in classical and noncommutative geometry as well as in applications to physics
(see, e.g., \cites{Connes,LSZ,CS} and references cited therein), and we prove perturbation results for such traces.

Let $\BH$ denote the algebra of bounded linear operators acting on a separable Hilbert space $\Hcal$, let $\Ic=\mathcal{L}^{(1,\infty)}(\BH)$ denote the dual Macaev ideal and  $\Tr_\omega$ a Dixmier trace on it (see Section \ref{sec:prelims} for details).
For certain classes of
pairs $(H_0,V)$, with $H_0$ an operator in $\BH$ or an unbounded operator affiliated with $\BH$ and with (bounded) $V\in\Ic$,
we prove the trace formula
\begin{align}
\label{tf}
\Tr_\omega\big(f(H_0+V)-f(H_0)\big)=\int_\Omega f'(\la)\,\mu_{H_0,V}(d\la),
\end{align}
whenever $f$ is a sufficiently nice scalar function, for some finite, complex measure $\mu_{H_0,V}$
on an appropriate subset $\Omega$ of the complex plane (see Theorem \ref{mt}, Remarks \ref{bo} and \ref{nr}, and Theorem \ref{ui}).

By analogy with the case of a normal trace, we call the measure $\mu_{H_0,V}$ the (first order) spectral shift measure.
If $H_0$ and $H_0+V$ are closed, densely defined  and dissipative
(i.e., possibly unbounded with 
$\im\left<H_0\xi,\xi\right>\geq 0$ and $\im\left<(H_0+V)\xi,\xi\right>\geq 0$ for every $\xi$ in the domain of $H_0$), then $\Omega$ can be taken to be $\Reals$;
if $H_0$ and $H_0+V$ are contractive, then $\Omega$ can be taken to be $\mathbb{T}$,
while if $H_0$ and $H_0+V$ are bounded self--adjoint, then $\Omega$ can be taken to be a bounded subset of $\Reals$.
In these various cases, different classes of functions $f$ are allowed in~\eqref{tf}.
Note that the
verbatim analog of \eqref{tf} for  general dissipative or contractive operators in case of a normal trace has not appeared in the literature (to the best of the authors' knowledge), but
is obtained in this paper, as our techniques and results apply more generally, and in particular to normal traces.

Singularity of the trace $\Tr_\omega$ entails properties of the spectral shift measures that do not hold for normal traces.
We demonstrate that the measure $\mu_{H_0,V}$ can fail to be absolutely continuous (see Proposition \ref{dm}) and, moreover, any measure type is possible (see Theorem \ref{prop:order1example}), as distinct from the case of trace class (or noncommutative $L^1$) perturbations and normal trace. We also show that the spectral shift measures linked to the Dixmier trace degenerate to zero when their counterparts linked to the standard trace are defined (see Proposition \ref{triv}).

We prove linearization of the operator function inside the trace
\begin{equation}
\label{lin}
\Tr_\omega\big(f(H_0+V)-f(H_0)\big)=\Tr_\omega\big(f'(H_0)V\big),
\end{equation}
with $V\in\Ic$ (see Theorem \ref{imt}), which does not hold in general for a normal trace
$\Tr$, for when $f$ is a nonlinear function, the Taylor series expansion for
$\Tr\big(f(H_0+V)-f(H_0)\big)$, with $V$ in the trace class, contains higher order G\^{a}teaux derivatives
$\frac{d^n}{dt^n}f(H_0+tV)$.
Note that although, seeing~\eqref{lin}, it is tempting to write $\mu_{H_0,V}(d\la)=\Tr_\omega\big(E_{H_0}(d\la)V\big)$, in general,
we do not have this equality because $\mu_{H_0,V}$ is a countably additive measure, while the set function
$\Tr_\omega\big(E_{H_0}(\cdot)V\big)$
can fail to be countably additive (see Example \ref{ex}).

For $V\in\Ic^{1/2}$, we prove the second order trace formula
\begin{align}
\label{2tf}
\Tr_\omega\left(f(H_0+V)-f(H_0)-\frac{d}{dt}\bigg|_{t=0}f(H_0+tV)\right)=\int_\Omega f''(\la)\,\nu_{H_0,V}(d\la),
\end{align}
where $\nu_{H_0,V}$ is a finite measure determined by the operators $H_0$ and $V$ (see Theorem \ref{2mt}).
The measure $\nu_{H_0,V}$ can fail to be absolutely continuous (see Proposition \ref{2dm})
and it degenerates to zero when $V\in\Ic$ (see Proposition \ref{2triv}).
Furthermore, we prove (see Theorem \ref{2imt})
\begin{multline}
\label{2lin}
\Tr_\omega\left(f(H_0+V)-f(H_0)-\frac{d}{dt}\bigg|_{t=0}f(H_0+tV)\right)=
\frac12\,\Tr_\omega\left(\frac{d^2}{dt^2}\bigg|_{t=0}f(H_0+tV)\right),
\end{multline}
for $V\in\Ic^{1/2}$,
which is an analogue of~\eqref{lin}.

Singularity of the trace requires
development of a new approach to the spectral shift measures, which is explained at the beginning of Section \ref{sec3}. We prove existence of the first and second order spectral shift measures implicitly, which is closer in spirit to the proof of existence of the higher order \cites{PSS,PSS-circle} than the proof of existence of the lower order spectral shift measures for normal traces.

In fact, analogous results, by the same proofs, hold in the more general setting of Dixmier traces on Marcinkiewicz ideals of
$\sigma$--finite, semifinite von Neumann algebra factors
(see~\cite{GI} and ~\cite{CS}),
and our exposition accommodates this generalization.
Our hypotheses on the trace are quite general, and accommodate also the classical trace.

The organization of the rest of the paper is as follows.
Section~\ref{sec:prelims} contains preliminaries, and is divided into three subsections: \ref{subsec:ideals} on general ideals,
norms and traces, \ref{subsec:Dixmier} on Dixmier traces and \ref{subsec:dilation} on classical dilation theory done for
semifinite von Neumann algebras.
Section~\ref{sec3} contains proofs of the first order perturbation formulas.
Section~\ref{notac} contains some results and examples on spectral shift measures $\mu_{H_0,V}$ for singular traces, showing, in particular, that $\mu_{H_0,V}$ can be singular. Section~\ref{sec:2ndOrd} contains proofs of the second order perturbation formulas.

\section{Preliminaries}
\label{sec:prelims}

\subsection{On ideals, norms and traces}
\label{subsec:ideals}

Let $\Bc$ be a $\sigma$--finite, semifinite von Neumann algebra factor with fixed normal, faithful, semifinite trace $\tau$.
We will consider perturbation formulas for possibly unbounded operators affiliated with $\Bc$.
While $\Bc$, when represented normally on a Hilbert space $\Hc$, consists entirely of bounded operators,
a densely defined, closed, (possibly) unbounded operator $T$ is affiliated with $\Bc$ if and only if $T$ commutes
with all unitary operators in the commutant of $\Bc$;
equivalently, $T$ is affiliated with $\Bc$
if and only if in the polar decomposition $T=V|T|$, we have $V\in\Bc$ and all spectral projections of $|T|$
belong to $\Bc$;
since an unbounded operator affiliated with $\Bc$ can be identified in this way as the product of
a partial isometry $V\in\Bc$
with an integral $|T|=\int_{[0,\infty)}t E_{|T|}(dt)$ of a $\Bc$--valued spectral measure, the algebra $\widetilde{\Bc}$
of operators affiliated with $\Bc$ is independent of the Hilbert space on which $\Bc$ is represented.
In the case of $\Bc=\BH$, the affiliated operators are, of course,
just arbitrary closed, densely defined, possibly unbounded operators on $\Hc$.

In the case of $\Bc=\BH$, von Neumann characterized the ideals of $\Bc$ in terms of the set of sequences
$(s_n(A))_{n=1}^\infty$ of singular numbers
of elements $A$ of the ideals (see~\cite{Calkin} or~\cite{GK}).
In the case of $\Bc$ a type II$_\infty$ factor with fixed normal, faithful trace $\tau$,
the ideals of $\Bc$ (and, more generally, the sub--$\Bc,\Bc$--bimodules of the space of all $\tau$--measurable
operators affiliated to $\Bc$)
are classified in terms of generalized singular numbers, which go back to Murray and von Neumann~\cite{MvN}.
For $A\in\Bc$, the generalized singular numbers $\mu_t(A)\ge0$ are a right--continuous, nonincreasing function of
$t\in(0,\infty)$.
See~\cite{FK} for more on generalized singular numbers.
Let us use the symbol $\mu(A)$ to denote the function $t\mapsto\mu_t(A)$.
By~\cite{GI}, an ideal $\Ic$ of $\Bc$ is characterized by its so called {\em characteristic set}
\[
\mu(\Ic)=\{\mu(A)\mid A\in\Ic\}
\]
and the sets of functions so arising are precisely the dilation invariant, hereditary cones
in the set of bounded right--continuous, nonincreasing functions.

\begin{defi}\label{def:idealnorm}
Let $\Bc$ be a $\sigma$--finite, semifinite von Neumann algebra factor.
An ideal $\Ic$ of $\Bc$ is called a {\em normed ideal} if it is equipped with an {\em ideal norm}, namely,
a norm $\|\cdot\|_\Ic$ on $\Ic$ satisfying
\begin{enumerate}[(i)]
\item
$A\in\Bc$, $B\in\Ic$, $0\le A\le B$ implies $\|A\|_\Ic\le\|B\|_\Ic$,
\item
there is a constant $K>0$ such that
$\|B\|\leq K\|B\|_\Ic$
for all $B\in\Ic$,
\item
for all $A,C\in\Bc$ and all $B\in\Ic$ we have
\[
\|ABC\|_\Ic\leq\|A\|\,\|B\|_\Ic\,\|C\|.
\]
\end{enumerate}
Likewise, $\Ic$ is a {\em quasi--normed ideal} if it has an {\em ideal quasi--norm}, which is a quasi--norm
on $\Ic$ satisfying properties (i)--(iii).
\end{defi}

\begin{defi}\label{def:trace}
A {\em trace} on an ideal $\Ic$ of $\Bc$ is a linear functional $\tau_\Ic:\Ic\to\mathbb{C}$ such that
\[
\tau_\Ic(AB)=\tau_\Ic(BA),\qquad(A\in\Ic,\,B\in\Bc).
\]
We say the trace is {\em positive} if
\[
A\in\Ic,\,A\ge0\quad\implies\quad\tau_\Ic(A)\ge0.
\]
If $\Ic$ has an ideal norm $\|\cdot\|_\Ic$,
we will say that a trace $\tau_\Ic$ is {\em $\|\cdot\|_\Ic$-- bounded} if there is a constant $M>0$ such that
\[
|\tau_\Ic(A)|\le M\|A\|_\Ic,\qquad(A\in\Ic).
\]
Note that the infimum of such constants $M$ equals $\|\tau_\Ic\|_{\Ic^*}$.
\end{defi}

Note that the purely algebraic
question of existence of traces on a given ideal $\Ic$ is solved in the discrete case (i.e., $\Bc=\BH$)
in~\cite{DFWW}
and in the continuous (i.e., $II_\infty$--factor) case in~\cite{DK}.
See~\cite{SZ} for results concerning existence of various sorts of traces on symmetrically normed ideals.

Our first order perturbation results will apply whenever we have a normed ideal $\Ic$
and positive trace $\tau_\Ic$ on it that is bounded with respect to the ideal norm.
As we will shortly see, Dixmier traces on Marcinkiewcz ideals provide examples of these.

By the characterizations of ideals in terms of (generalized) singular numbers, for every ideal $\Ic$ of $\Bc$ and
every $\alpha>0$, we have the ideal
\[
\Ic^\alpha=\{A\in\Bc\mid |A|^{1/\alpha}\in\Ic\}
\]
of $\Bc$
and by well known inequalities involving (generalized) singular numbers, we have that $\Ic^n$ for $n\in\Nats$
is spanned by the $n$--fold products of elements from $\Ic$.
Note that whenever we have a positive trace on $\Ic$,
we have the usual Cauchy-Schwarz inequality:
\begin{equation}\label{eq:CS}
|\tau_\Ic(AB)|\le\big(\tau_\Ic(|A|^2)\big)^{1/2}\big(\tau_\Ic(|B|^2)\big)^{1/2},\qquad(A,B\in\Ic^{1/2}).
\end{equation}

For second order perturbation results, we will ask that $\Ic^{1/2}$ be a normed ideal equipped with an ideal norm
$\|\cdot\|_{\Ic^{1/2}}$ such that
\begin{equation}\label{eq:I12}
\|AB\|_\Ic\le\|A\|_{\Ic^{1/2}}\|B\|_{\Ic^{1/2}},\qquad(A,B\in\Ic^{1/2}).
\end{equation}
Below we will give a criterion on Marcinkiewicz ideals that implies
the existence of such an ideal norm on $\Ic^{1/2}$.

\medskip
We conclude these preliminary remarks on ideals, norms and traces, with an easy result that will
be used in later proofs involving dilations of contractions and dissipative operators.

\begin{prop}\label{prop:extend}
Let $\Ic$ be a proper, nonzero ideal of a $\sigma$--finite, semifinite von Neumann algebra factor $\Bc$
and let $\Hc$ be a separable Hilbert space.
Let $\Nc$ be a $\sigma$--finite, semifinite von Neumann algebra factor containing
$\Bc$ as a corner, so that $\Bc$ is identified with $p\Nc p$ for a projection $p\in\Nc$.
\begin{enumerate}[(i)]
\item
There is an ideal $\Ict$ of $\Nc$ such that $\Ict\cap\Bc=\Ic$.
\item
If $\|\cdot\|_\Ic$ is an ideal norm on $\Ic$, then there is an ideal norm $\|\cdot\|_\Ict$
on $\Ict$ whose restriction to $\Ic$ equals $\|\cdot\|_\Ic$.
\item
If $\tau_\Ic$ is a trace on $\Ic$, then there is a trace $\tau_\Ict$ on $\Ict$ whose restriction to $\Ic$ is $\tau_\Ic$;
moreover, if $\tau_\Ic$ is positive, then $\tau_\Ict$ is positive, while if the hypothesis of~(ii) also holds and
if $\tau_\Ic$ is $\|\cdot\|_\Ic$--bounded,
then $\tau_\Ict$ is $\|\cdot\|_\Ict$--bounded.
\end{enumerate}
\end{prop}
\begin{proof}
Let $\Ict$ be the ideal of $\Nc$ whose characteristic set is $\mu(\Ic)$.
Then $\Ict\cap\Bc=\Ic$.

Since $\Bc$ has a proper, nonzero ideal, it is an infinite von Neumann algebra.
Consequently, $p$ is an infinite projection in $\Nc$.
Since $\Nc$ is a $\sigma$--finite but infinite factor, there is an isometry $v\in\Nc$ whose range is $p$.
Thus $x\mapsto vxv^*$ is a $*$--isomorphism from $\Nc$ onto $\Bc$.
Moreover, since for any $x\in\Nc$ we have $\mu(x)=\mu(vxv^*)$, we have $v\Ict v^*=\Ic$.
In case~(ii) we set $\|x\|_\Ict=\|vxv^*\|_\Ic$ while in case~(iii) we let $\tau_\Ict(x)=\tau_\Ic(vxv^*)$.
Now the assertions (ii) and (iii) are easily verified.
\end{proof}

\subsection{On Dixmier traces}
\label{subsec:Dixmier}

The trace introduced by Dixmier~\cite{Dix} is a singular
trace (i.e., non--normal in the technical sense of the term) on $\BH$.
The natural domain of definition of Dixmier's trace is
the dual Macaev ideal $\Ic=\Lc^{(1,\infty)}$ of $\BH$, which is the set of operators $A\in\BH$ such that
\[
\|A\|_\Ic:=\sup_{n\in\Nats}\frac1{\log(n+1)}\sum_{k=1}^ns_k(A)<\infty.
\]
If $\omega$ is a dilation invariant state on $\ell^\infty(\Nats)$, then $\Tr_\omega:\Ic\to\mathbb{C}$
defined by
\begin{equation}\label{eq:trB}
\Tr_\omega(A)=\omega\left(\frac1{\log(n+1)}\sum_{k=1}^ns_k(A)\right)
\end{equation}
for $A\ge0$
is a positive trace on $\Ic$ that is bounded with respect to the ideal norm.
Note that $\Tr_\omega$ vanishes on all finite rank operators.

\medskip
More generally, we will consider Dixmier traces on Marcinkiewicz (also called Lorentz) ideals associated to $\sigma$--finite, semifinite von Neumann algebra factors.
(See~\cite{GI}, \cite{CS} and references therein, and note that we are concerned only with the Dixmier traces supported
at $\infty$, as described in~\cite{CS}.)
Let $\psi$ be
a concave function satisfying
\[
\lim_{t\rightarrow 0^+}\psi(t)=0, \qquad \lim_{t\rightarrow \infty}\psi(t)=\infty.
\]
Then the Marcinkiewicz ideal
$\Ic=\Ic_\psi$ of $\Bc$ and its ideal norm $\|\cdot\|_\Ic$ are defined by
\begin{equation}\label{eq:AMnorm}
\Ic=\bigg\{A\in\Bc\,\bigg|\,\|A\|_\Ic:=\sup_{t>0}\frac1{\psi(t)}\int_0^t\mu_s(A)\,ds<\infty\bigg\}.
\end{equation}

\begin{remark}
The discrete case, namely
Marcinkiewicz ideals of $\Bc=\BH$, can formally be included in the above formalism that applies in the continuous case.
Given a function $\psi:\Nats\to(0,\infty)$ satisfying $\psi(n)+\psi(n+2)\le2\psi(n+1)$ and $\lim_{n\to\infty}\psi(n)=\infty$,
by first adding a constant to $\psi$, if necessary, so that we may define $\psit(0)=0$ and still have
$\psit(0)+\psit(2)\le2\psit(1)$,
we may then
extend $\psi$ to a concave function $\psit:[0,\infty)\to[0,\infty)$ by
piecewise linear interpolations.
Now, as the generalized singular numbers of elements of $\BH$ are constant on intervals $[n-1,n)$ for $n\in\Nats$
and are zero on $[0,1)$, we find that the definition of the norm in~\eqref{eq:AMnorm}
is equivalent to
\[
\|A\|_\Ic=\sup_{n\in\Nats}\frac1{\psi(n)}\sum_{k=1}^ns_k(A).
\]
\end{remark}

\begin{prop}\label{prop:Malpha}
Let $\Ic=\Ic_\psi$ be the Marcinkiewicz ideal of $\Bc$ described by~\eqref{eq:AMnorm}.
For $p>0$, the ideal $\Ic^{1/p}$ is equipped with the ideal quasi--norm
\begin{equation}\label{eq:Malphanorm}
\|B\|_{\Ic^{1/p}}=\left(\|\,|B|^p\,\|_\Ic\right)^{1/p},
\end{equation}
which is a norm when $p\ge1$ and a $p$--quasi--norm when $0<p<1$.
If $p>1$, then letting $q$ be such that $\frac1p+\frac1q=1$, we have the H\"older--like inequality
\begin{equation}\label{eq:IHolder}
\|A B\|_{\Ic}\le\|A\|_{\Ic^{1/p}}\,\|B\|_{\Ic^{1/q}}\qquad(A\in\Ic^{1/p},\,B\in\Ic^{1/q}).
\end{equation}
\end{prop}
\begin{proof}
The fact that~\eqref{eq:Malphanorm} defines a norm when $p\ge1$
follows from the Minkowski inequality for generalized singular numbers, (\cite{F} Cor. 4.4(i)),
while if $0<p<1$, then from Thm.~4.7(i) of~\cite{FK} we have
\[
\|A+B\|_{\Ic^{1/p}}^p\le\|A\|_{\Ic^{1/p}}^p+\|B\|_{\Ic^{1/p}}^p,\qquad(A,B\in\Ic^{1/p}).
\]
Therefore, $\|\cdot\|_{\Ic^{1/p}}$ is a $p$--quasi--norm.

Now the conditions (i)--(iii) in Definition~\ref{def:idealnorm} follow for $\|\cdot\|_{\Ic^{1/p}}$
from well known properties of (generalized) singular numbers.
For the H\"older inequality~\eqref{eq:IHolder}, by Corollary 4.4(iii) of~\cite{F}, we have
\begin{align*}
\frac1{\psi(t)}\int_0^t\mu_s(AB)\,ds
&\le\left(\frac1{\psi(t)}\int_0^t\mu_s(A)^p\,ds\right)^{1/p}
 \left(\frac1{\psi(t)}\int_0^t\mu_s(B)^q\,ds\right)^{1/q} \\[1ex]
&\le\|A\|_{\Ic^{1/p}}\|B\|_{\Ic^{1/q}},
\end{align*}
so taking the supremum over all $t>0$ yields the desired inequality.
\end{proof}

We now describe Dixmier traces on Marcinkiewicz ideals.
By Theorem~2.2 of~\cite{CS} (see also~\cite{DPSSS} and \cite{GI}),
if
\[
\liminf_{t\to\infty}\frac{\psi(2t)}{\psi(t)}=1,
\]
then a positive, $\|\cdot\|_\Ic$--bounded trace $\tau_\Ic$ on $\Ic$ can be constructed analogously to~\eqref{eq:trB},
by, for $A\in\Ic$, $A\ge0$, taking $\tau_\Ic(A)$ to be a dilation invariant Banach limit on $L^\infty((0,\infty))$ evaluated at
\begin{equation}\label{eq:int}
\frac1{\psi(t)}\int_0^t\mu_s(B)\,ds
\end{equation}
as $t\to\infty$.
A trace constructed in this way is called a {\em Dixmier trace}.

\medskip
Some of our results will apply when $\tau_\Ic$ vanishes on $\Ic^\alpha$, for certain $\alpha>1$.
\begin{prop}
Consider a Marcinkiewicz ideal $\Ic=\Ic_\psi$ and
let $\tau_\Ic$ be a Dixmier trace on it.
Suppose that for some $0<\eps<1$ there is a constant $C$ such that
\begin{equation}\label{eq:KSScond}
\psi(s)\le Cs^\eps,\qquad(s\ge1).
\end{equation}
Then for all $\alpha>1/(1-\eps)$, we have $\tau_\Ic(\Ic^\alpha)=\{0\}$.
\end{prop}
\begin{proof}
It will suffice to show: if $A\in\Ic^\alpha$ and $A\ge0$, then $\tau_\Ic(A)=0$.
Since $A^{1/\alpha}\in\Ic$, for $t\ge1$ we have
\[
t\mu_t(A)^{1/\alpha}=t\mu_t(A^{1/\alpha})\le\int_0^t\mu_s(A^{1/\alpha})\,ds
\le\psi(t)\|A^{1/\alpha}\|_\Ic\le C_\eps t^\eps\|A^{1/\alpha}\|_\Ic
\]
and we have
\[
\mu_t(A)\le Dt^{-\alpha(1-\eps)},\qquad(t\ge1)
\]
for a constant $D$ independent of $t$.
Consequently, the function $s\mapsto\mu_s(A)$ is integrable, and we have
\[
\lim_{t\to\infty}\frac1{\psi(t)}\int_0^t\mu_s(A)\,ds=0,
\]
which ensures $\tau_\Ic(A)=0$.
\end{proof}

Note that, by the usual sorts of estimates,
the hypotheses involving~\eqref{eq:KSScond} of the above lemma are satisfied for all $\eps>0$
if we have
\[
\lim_{t\to\infty}\frac{\psi(2t)}{\psi(t)}=1.
\]
Indeed, for any $C>1$ and $t$ large enough, we have $\psi(2t)\le C\psi(t)$,
so for some $t_0$ and all $n\in\Nats$, $\psi(2^nt_0)\le C^n\psi(t_0)$.
Since $\psi$ is increasing, we get
\[
\limsup_{t\to\infty}\frac{\log(\psi(t))}{\log t}\le\frac{\log C}{\log2}<\eps.\
\]
for $C$ sufficiently close to $1$.
Thus, for example, we have $\Tr_\omega(\Ic^\alpha)=\{0\}$ for all $\alpha>1$, whenever $\Tr_\omega$
is a Dixmier trace defined as in~\eqref{eq:trB} on $\Ic=\Lc^{(1,\infty)}\subseteq\BH$.

\subsection{Dilation theory for operators in semifinite von Neumann algebras}
\label{subsec:dilation}

In this section we make some observations about the classical Sz.-Nagy dilation results (see~\cite{SNF}) for contractions
and (possibly unbounded) dissipative operators, that
are pertinent when working in semifinite von Neumann algebras.
Recall that a {\em unitary dilation} of a contraction $T\in\BH$ is a unitary $U\in\BK$ for a Hilbert space $\Kc$
containing $\Hc$ as a closed subspace, such that $T^n=pU^n\restrict_\Hc$ for every $n\in\Nats$,
where $p$ is the orthogonal projection from $\Kc$ onto $\Hc$.

\begin{prop}\label{prop:Udil}
Let $\Bc$ be a semifinite (or finite) von Neumann algebra with normal, faithful, semifinite (or finite) trace
$\tau$.
Let $T\in\Bc$ be a contraction.
Then there is a semifinite von Neumann algebra $\Nc$ with normal faithful, semifinite trace $\tau_\Nc$ and a normal
inclusion $\Bc\hookrightarrow\Nc$ sending the identity element $I_\Bc$ to a projection $p\in\Nc$, and there is a unitary
element $U\in\Nc$ such that
\begin{enumerate}[(a)]
\item
the restriction of $\tau_\Nc$ to the positive elements of $\Bc$ agrees with $\tau$,
\item
$p\Nc p=\Bc$ and the central support of $p$ in $\Nc$ is $I_\Nc$,
\item
for all $n\in\Nats$,
\[
T^n=pU^np.
\]
\end{enumerate}
\end{prop}
\begin{proof}
We simply follow the proof contained in the first part of Chapt.\ I, Sec.\ 5 of~\cite{SNF}.
If $\Bc$ is normally unitally represented on a Hilbert space $\Hc$, then let
\[
\Nc=\Bc\overline{\otimes} B(\ell^2(\Z))\subseteq B(\Hc\otimes\ell^2(\Z)).
\]
Writing $(e_{ij})_{i,j\in\Z}$ for the usual system of matrix units in $B(\ell^2(\Z))$, we identify $\Bc$
with $\Bc\otimes e_{00}$ and have $p=I_\Bc\otimes e_{00}$.
Let
\begin{equation}\label{eq:U}
U=T\otimes e_{00}+D_T\otimes e_{-1,0}+D_{T^*}\otimes e_{01}-T^*\otimes e_{-1,1}+\sum_{i\in\Z\backslash\{-1,0\}}I_\Bc\otimes e_{i,i+1}\,,
\end{equation}
where $D_T=(I_\Bc-T^*T)^{1/2}$ and $D_{T^*}=(I_\Bc-TT^*)^{1/2}$ are the defect operators.
Then $U$ is unitary and the desired properties hold.
\end{proof}

For future use, we now prove:
\begin{lemma}\label{lem:eval1}
For a semifinite von Neumann algebra $\Bc$,
if $T\in\Bc$ is a contraction not having eigenvalue $1$ and if $U$ is the unitary dilation of $T$ from Proposition~\ref{prop:Udil},
then $U$ does not have eigenvalue $1$.
\end{lemma}
\begin{proof}
Suppose $\xi\in\Hc\otimes\ell^2(\Z)$ and $U\xi=\xi$, and let us show $\xi$ must be $0$.
We write $\xi=\sum_{n\in\Z}\xi_n\otimes\delta_n$ for $\xi_n\in\Hc$ and $\delta_n\in\ell^2(\Z)$ the characteristic
function of $\{n\}$.
Using $U\xi=\xi$ and~\eqref{eq:U} we get $\xi_i=\xi_{i+1}$ for all
$i\in\Z\backslash\{-1,0\}$.
Since $\|\xi\|^2=\sum_{i\in\Z}\|\xi_i\|^2$ is finite, we must have $\xi_i=0$ for $i\ne0$.
But then we must have $T\xi_0=\xi_0$, and by hypothesis this implies $\xi_0=0$.
So $\xi=0$.
\end{proof}

\medskip
Now we turn to self--adjoint dilations of dissipative operators.
This is a well understood theory, but we will run through some rudimentary parts of it, in order
to do it in the setting of semifinite von Neumann algebras.

A dissipative operator $A$ on a Hilbert space $\Hc$ is a densely defined, possibly unbounded operator on $\Hc$
satisfying $\im\langle A\xi,\xi\rangle\ge0$ for all $\xi$ in the domain of $A$.
The basic theory of dissipative operators can be found in Chapt.\ IV, Sec.\ 4 of~\cite{SNF}.
It includes that every dissipative operator $A_0$ has a maximal dissipative operator extension $A$.
Henceforth, we will use the term dissipative operator to mean a maximal dissipative operator.

A {\em self--adjoint dilation} of a dissipative operator on $\Hc$ is a self--adjoint, densely defined, possibly unbounded operator
$X$ on a Hilbert space $\Kc$ that contains $\Hc$ as a closed subspace and such that, for all $z\in\Complex$ with $\im\,z<0$,
and every $k\in\Nats$, we have
\begin{equation}\label{eq:AX}
(A-zI_\Hc)^{-k}=p(X-zI_{\Kc})^{-k}\restrict_\Hc\,.
\end{equation}

The Cayley transform of a maximal dissipative operator $A$ is the unique contraction $T\in\BH$
(not having eigenvalue $1$)
such that
\[
A=\ii(I_\Hc+T)(I_\Hc-T)^{-1}.
\]
(See Chapt.~IV, Sec.~4 of~\cite{SNF}.)
Standard calculations then show that, for $z\in\Complex$ with $\im\,z<0$,
\begin{equation}\label{eq:AzI}
(A-zI_\Hc)^{-1}=(\ii-z)^{-1}(I_\Hc-T)\sum_{k=0}^\infty\left(\frac{z+\ii}{z-\ii}\right)^kT^k\,.
\end{equation}
For every $k\in\Nats$, taking the $k$-th power of~\eqref{eq:AzI} then yields
\[
(A-zI_\Hc)^{-k}=\sum_{n=0}^\infty w_n(k,z)\,T^n
\]
for complex numbers $w_n(k,z)$ such that $\sum_n|w_n(k,z)|<\infty$.
If $U\in\BK$ is a unitary dilation of $T$ without eigenvalue $1$, then taking the self--adjoint (possibly
unbounded) operator $X=\ii(I_\Kc+U)(I_\Kc-U)^{-1}$, we have
\[
(X-zI_\Hc)^{-k}=\sum_{n=0}^\infty w_n(k,z)\,U^n\qquad(\im\,z<0,\,k\in\Nats)
\]
and we see that $X$ is a self--adjoint dilation of $A$.
Employing this procedure with the unitary dilation obtained from Proposition~\ref{prop:Udil} and using
Lemma~\ref{lem:eval1} to see that it has no eigenvalue $1$, we get:
\begin{prop}\label{prop:Adil}
Let $\Bc$ be a semifinite (or finite) von Neumann algebra with normal, faithful, semifinite (or finite) trace
$\tau$.
Let $A$ be a dissipative operator affiliated to $\Bc$.
Then there is a semifinite von Neumann algebra $\Nc$ with normal faithful, semifinite trace $\tau_\Nc$ and a normal
inclusion $\Bc\hookrightarrow\Nc$ sending the identity element $I_\Bc$ to a projection $p\in\Nc$
whose central support in $\Nc$ is $I_\Nc$, and there is a
self--adjoint operator $X$ affiliated to $\Nc$ such that
the restriction of $\tau_\Nc$ to the positive elements of $\Bc$ agrees with $\tau$,
$p\Nc p=\Bc,$
and~\eqref{eq:AX} holds
for all $z\in\Complex$ with $\im\,z<0$ and all $k\in\Nats$.
\end{prop}

\section{Existence of spectral shift measures}
\label{sec3}

We start by recollecting some main ideas used in the proofs of existence of the (first order) spectral shift measures for
the normal trace $\Tr$ on $\BH$
and self-adjoint or unitary operators.

In the original proof of Krein~\cites{Krein,Kreinunitary}, existence of the absolutely continuous spectral shift measures was
first
established for finite rank perturbations and then transferred to trace class perturbations by approximations. This approach is
not applicable to singular traces.

Another proof of existence of the spectral shift measures was derived in \cite{BS} via double operator integration.
In case of self-adjoint operators $H_0$, $V$, with $V$ in the trace class, the spectral shift measure is given
by the explicit formula
\[\mu_{H_0,V}(d\la)=\int_0^1\Tr\big(E_{H_0+tV}(d\la)V\big)\,dt\]
(see \cite{BS}), which is derived as follows:
\begin{align*}
&\Tr\big(f(H_0+V)-f(H_0)\big)=\int_0^1\Tr\left(\frac{d}{dt} f(H_0+tV)\right)\,dt\\
&\quad=\int_0^1\int_\Reals f'(\la)\Tr\big(E_{H_0+tV}(d\la)V\big)\,dt
=\int_\Reals f'(\la)\left(\int_0^1\Tr\big(E_{H_0+tV}(d\la)V\big)\,dt\right).
\end{align*}
Change of the order of integration above is justified by boundedness of $f'$, finiteness of the measure $\Tr\big(E_{H_0+tV}(\cdot)V\big)$, and measurability of the function $t\mapsto\Tr\big(E_{H_0+tV}(d\la)V\big)$.
Measurability of $t\mapsto\Tr\big(E_{H_0+tV}(d\la)V\big)$ (see, e.g., \cite{Azamov}*{Lemma 6.2}) relies on the following continuity property of the normal trace (see, e.g., \cite{Azamov}*{Lemma 2.5}): if $\{A_\alpha\}_\alpha$ is a uniformly bounded net of operators in $\BH$ converging in the strong operator topology to $A\in\BH$, and $V$ is in the trace class, then $\{\Tr(A_\alpha V)\}_\alpha$ converges to $\Tr(AV)$.

We now show that, in case of a Dixmier trace $\Tr_\omega$ defined on the Marcinkiewicz ideal $\Ic=\Lc^{(1,\infty)}$
of $\BH$, the finitely additive measure $\Tr_\omega\big(E_{H_0}(\cdot)V\big)$ can fail to be countably additive.

\begin{example}
\label{ex}
If $H_0$ is an unbounded self-adjoint operator with discrete spectrum (that is, the spectrum of $H_0$ consists of isolated eigenvalues of finite multiplicities), then $E_{H_0}(\Delta)$ is a finite rank projection whenever $\Delta$ is a bounded interval and $E_{H_0}(\Reals)=I$. Hence,
\begin{align*}
\Tr_\omega\big(E_{H_0}(\Delta)V\big)=
\begin{cases}
0 & \text{ if } \Delta \text{ is a bounded interval}\\
\Tr_\omega(V) & \text{ if } \Delta=\Reals.
\end{cases}
\end{align*}
\end{example}


For the remainder of this section, we assume the following.
\begin{hyp}
\label{ideal}
Let $\Ic$ be a normed ideal of a $\sigma$--finite, semifinite von Neumann algebra factor $\Bc$,
with ideal norm denoted $\|\cdot\|_\Ic$, and
endowed with a trace $\tau_\Ic:\Ic\to\Complex$ that is positive and $\|\cdot\|_\Ic$--bounded.
\end{hyp}
Note that the Dixmier trace $\tau_\Ic=\Tr_\omega$, with $\Bc=\BH$ and $\Ic=\Lc^{(1,\infty)}$, satisfies Hypotheses \ref{ideal}.
Also, 
for a semifinite von Neumann algebra $\Bc$ with normal, faithful, semifinite trace $\tau$,
taking
$\Ic=\{A\in\Bc:\, \tau(|A|)<\infty\}$ equipped with the norm $\|A\|_\Ic=\max\{\|A\|,\tau(|A|)\}$ and the trace $\tau_\Ic=\tau$,
Hypotheses \ref{ideal} are satisfied.

We will prove the trace formula \eqref{tf} (or rather, its generalization changing $\Tr_\omega$ to $\tau_\Ic$)
under any of the following assumptions (see also Remark \ref{nr} and Theorem \ref{ui}).

\begin{hyp}
\label{cases}
A set $\Omega$, a closed, densely defined operator $H_0$ affiliated to $\Bc$, an operator $V\in\Ic$
and a space $\Fc$ of functions are taken to satisfy any of the following assertions.
\begin{enumerate}[(i)]
\item \label{i} $\Omega=\operatorname{conv}\big(\sigma(H_0)\cup\sigma(H_0+V)\big)$, $H_0=H_0^*\in\Bc$, $V=V^*$,
$\Fc=C^3(\Reals)$;
\item \label{ii}$\Omega=\Reals$, $H_0$ and $H_0+V$ are dissipative, and
\[
\Fc=\text{\rm span}\left\{\la\mapsto (z-\la)^{-k}:\, k\in\Nats,\, \im(z)<0\right\};
\]
\item \label{iii} $\Omega=\mathbb{T}$, $\|H_0\|\leq 1$, $\|H_0+V\|\leq 1$, and $\Fc$ is the set of all functions that are analytic on discs centered at $0$ and of radius strictly larger than $1$.
\end{enumerate}
\end{hyp}

\begin{thm}
\label{mt} Assume Hypotheses \ref{ideal}.
Let $\Omega$, $H_0$, $V$ and $\Fc$ satisfy Hypotheses \ref{cases}.
Then, there exists a (countably additive, complex)
measure $\mu_{H_0,V}$ on $\Omega$ such that for all $f\in\Fc$, the trace formula
\begin{align}
\label{tf2}
\tau_\Ic\big(f(H_0+V)-f(H_0)\big)=\int_\Omega f'(\la)\,\mu_{H_0,V}(d\la)
\end{align}
holds.
Moreover, the total variation of $\mu_{H_0,V}$
is bounded as follows:
\begin{align}
\label{tfe}
\big\|\mu_{H_0,V}\big\|\leq \tau_\Ic\big(|\re(V)|\big)+\tau_\Ic\big(|\im(V)|\big).
\end{align}
If Hypotheses \ref{cases}(i) are satisfied, then the measure $\mu_{H_0,V}$ is real and unique.
\end{thm}

\begin{remark}
\label{bo}
By applying Theorem~\ref{mt} under Hypotheses~\ref{cases}(ii) and then taking complex conjugates,
or by rescaling the operators and applying the theorem under Hypotheses~\ref{cases}(iii),
we also get the theorem under either of the following hypotheses:
\begin{itemize}
\item[(iv)] $\Omega=\Reals$, $H_0=H_0^*$ (possibly unbounded) and $V=V^*$, with
\[
\Fc=\text{\rm span}\left\{\la\mapsto (z-\la)^{-k}:\, k\in\Nats,\, \im(z)\ne0\right\};
\]
\item[(v)]
$H_0\in\Bc$, $\Omega=a\mathbb{T}$ for any $a\ge\max(\|H_0\|,\|H_0+V\|)$
and $\Fc$ the set of all functions that are
analytic on discs centered at $0$ and of radius strictly larger than $a$.
\end{itemize}
\end{remark}

Before proceeding to the proof, let us explain some assertions of the above hypotheses.

(1) The operator function $f(H)$ is determined by the values of a scalar function $f$ on the spectrum of $H=H^*$.
Hence in~\ref{cases}(i), $f(H)=g(H)$, for any $g\in C_c^3(\Reals)$ that agrees with $f$ on $\sigma(H)$
and, without weakening the results of the paper, we will prove the formula in this case only for $f\in C_c^3(\Reals)$.
Further comments on the set $\Fc$ of functions will be made in Remark \ref{l4}.

(2) The condition of both $H_0$ and $H_0+V$ being dissipative (or contractive) is imposed to make sure that the path  $H_0+tV=(1-t)H_0+t(H_0+V)$, $t\in[0,1]$, consists entirely of dissipative operators (or contractions).

The following lemmas are building blocks for the proof of existence of the spectral shift measure and the trace formula \eqref{tf2}.
The first of these is routine.

\begin{lemma}\label{l1}  Assume Hypothesis \ref{ideal}.
\begin{enumerate}[(i)]
\item \label{l1i}
For $H_0$ an operator affiliated with $\Bc$ and for $V\in\Bc$,
any $k\in\Nats$,  $t_0\in [0,1]$, and any $z\in\Complex$ such that
\begin{equation}\label{eq:supz}
\sup\limits_{t\in [0,1]}\|(zI-H_0-tV)^{-1}\|<\infty,
\end{equation}
we have
\begin{align*}
&(zI-H_0-V)^{-k}-(zI-H_0)^{-k}=\sum_{\substack{1\leq k_0,k_1\leq k\\k_0+k_1=k+1}}(zI-H_0-V)^{-k_0}V(zI-H_0)^{-k_1},\\
&\frac{d}{dt}\bigg|_{t=t_0}\big((zI-H_0-tV)^{-k}\big)=\sum_{\substack{1\leq k_0,k_1\leq k\\k_0+k_1=k+1}} (zI-H_0-t_0V)^{-k_0}V(zI-H_0-t_0V)^{-k_1}.
\end{align*}
Note that when $H_0$ and $H_0+V$ are dissipative, then~\eqref{eq:supz} holds whenever $\im(z)<0$,
while if $H_0=H_0^*$, $V=V^*$, then $z$ can be taken to be any complex number with $\im(z)\neq 0$;
\item \label{l1ii}
Let $H_0,V\in\Bc$. Then, for $k\in\Nats$  and $t_0\in [0,1]$,
\begin{align*}
(H_0+V)^k-H_0^k&=\sum_{\substack{0\leq k_0,k_1\\k_0+k_1=k-1}}(H_0+V)^{k_0}VH_0^{k_1},\\
\frac{d}{dt}\bigg|_{t=t_0}\big((H_0+tV)^k\big)&=\sum_{\substack{0\leq k_0,k_1\\k_0+k_1=k-1}} (H_0+t_0V)^{k_0}V(H_0+t_0V)^{k_1}.
\end{align*}
\end{enumerate}
\end{lemma}

If $H_0$ is bounded, then G\^{a}teaux derivatives $\frac{d}{dt}f(H_0+tV)$ can be computed for more general scalar functions $f$.
Let $W_n$ denote the set of functions $f\in C^n(\Reals)$ such that $f^{(j)},\widehat{f^{(j)}}\in L^1(\Reals)$, for $j=0,\ldots,n$. The set $W_n$ includes $C_c^{n+1}(\Reals)$ and $\text{\rm span}\left\{\Reals\ni\la\mapsto (z-\la)^{-k}:\, k\in\Nats,\, \im(z)\neq 0\right\}$.

\begin{lemma}\label{l1a} Assume Hypothesis \ref{ideal}.
\begin{enumerate}[(i)]
\item \label{l1ai}
Let $H_0=H_0^*\in\Bc$ and $V=V^*\in\Ic$. Then, for every $f\in W_1$,
\[f(H_0+V)-f(H_0)=\frac{\ii}{\sqrt{2\pi}}\int_\Reals\int_0^\la e^{\ii(\la-x)(H_0+V)}Ve^{\ii x H_0}\hat{f}(\la)\,dx\,d\la\]
and,  for every $t_0\in [0,1]$,
\[\frac{d}{dt}\bigg|_{t=t_0}f(H_0+tV)=\frac{\ii}{\sqrt{2\pi}}\int_\Reals\int_0^\la e^{\ii(\la-x)(H_0+t_0V)}Ve^{\ii x (H_0+t_0V)}\hat{f}(\la)\,dx\,d\la,\]
with the Bochner integrals evaluated in the ideal norm $\|\cdot\|_\Ic$.

\item \label{l1aii}
Let $H_0\in\Bc$, $\|H_0\|\leq 1$, $V\in\Ic$, and $\|H_0+V\|\leq 1$. Then, for every $f$ analytic on a disc of radius $r>1$ centered at $0$,
\[f(H_0+V)-f(H_0)=\sum_{k=1}^\infty \hat f(k)\sum_{\substack{0\leq k_0,k_1\\k_0+k_1=k-1}}(H_0+V)^{k_0}VH_0^{k_1}\]
and,  for every $t_0\in [0,1]$,
\[\frac{d}{dt}\bigg|_{t=t_0}f(H_0+tV)=\sum_{k=1}^\infty \hat f(k) \sum_{\substack{0\leq k_0,k_1\\k_0+k_1=k-1}}(H_0+t_0V)^{k_0}V(H_0+t_0V)^{k_1},\] where both series converge in $\|\cdot\|_\Ic$.
\end{enumerate}
\end{lemma}

\begin{proof}
(i) By the functional calculus,  for $H=H^*\in B(\mathcal{H})$ and $y,y'\in\Reals$,
\[\|e^{\ii y H}-e^{\ii y' H}\|\leq |y-y'|\cdot\|H\|.\]
Therefore, the function
\[x\mapsto e^{\ii(\la-x)(H_0+V)}Ve^{\ii x H_0}\] is uniformly continuous in the ideal norm $\|\cdot\|_\Ic$ and,  hence,  we have Duhamel's formula (see, e.g., \cite{Azamov}*{Lemma 5.2})
\[e^{\ii\la(H_0+V)}-e^{\ii\la H_0}=\ii\int_0^\la e^{\ii(\la-x)(H_0+V)}Ve^{\ii x H_0}\,dx,\] with the integral converging in the norm $\|\cdot\|_\Ic$.
Likewise, the function \[(x,\la)\mapsto e^{\ii(\la-x)(H_0+V)}Ve^{\ii x H_0}\] is uniformly continuous. Since $\widehat{f'}\in L^1(\Reals)$, by the spectral theorem and Fourier inversion,
\begin{align*}
f(H_0+V)-f(H_0)&=\frac{1}{\sqrt{2\pi}}\int_\Reals\left(e^{\ii\la(H_0+V)}-e^{\ii\la H_0}\right)\hat f(\la)\,d\la\\
&=\frac{\ii}{\sqrt{2\pi}}\int_\Reals\int_0^\la e^{\ii(\la-x)(H_0+V)}Ve^{\ii x H_0}\hat f(\la)\,dx\,d\la,
\end{align*}
where the multiple Bochner integral converges in $\|\cdot\|_\Ic$.

By the same method as above,
\begin{align*}
&\frac{f(H_0+tV)-f(H_0+t_0V)}{t-t_0}=\frac{\ii}{\sqrt{2\pi}}\int_\Reals\int_0^\la e^{\ii(\la-x)(H_0+tV)}Ve^{\ii x (H_0+t_0V)}\hat f(\la)\,dx\,d\la,\\
&\big\|e^{\ii(\la-x)(H_0+tV)}-e^{\ii(\la-x)(H_0+t_0V)}\big\|\leq|t-t_0|\cdot|\la-x|\cdot\|V\|.
\end{align*}
Thus, we have
\[\lim_{t\rightarrow 0}\big\|e^{\ii(\la-x)(H_0+tV)}Ve^{\ii x(H_0+t_0V)}-e^{\ii(\la-x)(H_0+t_0V)}Ve^{\ii x(H_0+t_0V)}\big\|_\Ic=0,\]
for all $(\la,x)\in \{(s_0,s_1)\in\Reals^2:\, |s_1|\leq |s_0|,\, \sign(s_0)=\sign(s_1)\}$.  Since $\widehat{f'}\in L^1(\Reals)$, by the Lebesgue dominated convergence theorem for Bochner integrals,
\begin{align}
\frac{d}{dt}\bigg|_{t=t_0}f(H_0+tV)=\frac{\ii}{\sqrt{2\pi}}\int_\Reals\int_0^\la e^{\ii(\la-x)(H_0+t_0V)}Ve^{\ii x (H_0+t_0V)}\hat f(\la)\,dx\, d\la,
\end{align}
where the integral converges in $\|\cdot\|_\Ic$.

(ii) Note that $f(z)=\sum_{k=0}^\infty \hat f(k)z^k$ and $f'(z)=\sum_{k=0}^\infty k\hat f(k)z^{k-1}$, where the series converge absolutely for $z\in\overline{\mathbb{D}}$, the closed unit disc centered at $0$. Thus, the formula for the operator derivative follows from the representation for the difference of operator monomials in Lemma \ref{l1} \eqref{l1ii} and convergence of the series $\sum_{k=1}^\infty|\hat f(k)|$ and $\sum_{k=1}^\infty|k\hat f(k)|$.
\end{proof}

\begin{lemma}
\label{l6} Assume Hypothesis \ref{ideal}.
Let $H_0$, $V$, and $\Fc$ satisfy Hypotheses \ref{cases}. Then for all $f\in\Fc$,
\begin{align*}
\tau_\Ic\left(\frac{d}{dt} f(H_0+tV)\right)=\frac{d}{dt}\,\tau_\Ic\big(f(H_0+tV)-f(H_0)\big).
\end{align*}
\end{lemma}

\begin{proof}
Assume first that Hypotheses \ref{cases}\eqref{i} are satisfied.
By $\|\cdot\|_\Ic$--boundedness of the trace $\tau_\Ic$ and Lemma \ref{l1a} (both the result and the method of the proof),
\begin{align*}
&\frac{d}{dt}\bigg|_{t=t_0}\,\tau_\Ic\big(f(H_0+tV)-f(H_0)\big)\\
&\quad=\lim_{t\rightarrow t_0}\frac{\ii}{\sqrt{2\pi}}\int_\Reals\int_0^\la \tau_\Ic\big(e^{\ii(\la-x)(H_0+tV)}Ve^{\ii x (H_0+t_0 V)}\big)\hat f(\la)\,dx\,d\la\\
&\quad=\frac{\ii}{\sqrt{2\pi}}\int_\Reals\int_0^\la \tau_\Ic\big(e^{\ii(\la-x)(H_0+t_0V)}Ve^{\ii x (H_0+t_0 V)}\big)\hat f(\la)\,dx\,d\la\\
&\quad=\tau_\Ic\left(\frac{d}{dt}\bigg|_{t=t_0} f(H_0+tV)\right).
\end{align*}

Assume now that Hypotheses \ref{cases}\eqref{ii} are satisfied. It is enough to prove the lemma for
$f(\la)=(z-\la)^{-k}$, $k\in\Nats$,
in which case we have
\begin{align*}
&\frac{d}{dt}\bigg|_{t=t_0}\,\tau_\Ic\big(f(H_0+tV)-f(H_0)\big)\\
&\quad=\lim_{t\rightarrow t_0}\tau_\Ic\left(\frac{f(H_0+tV)-f(H_0+t_0 V)}{t-t_0}\right)\\
&\quad =\lim_{t\rightarrow t_0}\tau_\Ic\left(\sum_{\substack{1\leq k_0,k_1\leq k\\k_0+k_1=k+1}}(zI-H_0-tV)^{-k_0}V(zI-H_0-t_0 V)^{-k_1}\right),
\end{align*}
which by the resolvent identity and $\|\cdot\|_\Ic$--boundedness of $\tau_\Ic$ equals
\[\tau_\Ic\left(\sum_{\substack{1\leq k_0,k_1\leq k\\k_0+k_1=k+1}}(zI-H_0-t_0 V)^{-k_0}V(zI-H_0-t_0 V)^{-k_1}\right).\]
Application of Lemma \ref{l1} completes the proof of the lemma assuming \ref{cases}\eqref{ii}.

If Hypotheses \ref{cases}\eqref{iii} are satisfied, then the lemma can be proved similarly to Lemma \ref{l1a}\eqref{l1aii}.
\end{proof}

\begin{lemma}\label{l3} Assume Hypothesis \ref{ideal}.
Let $H$ be a normal operator affiliated to $\Bc$
and $V\in\Ic$.
Then, for any finite Borel partition $\{\delta_i\}_{i=1}^n$ of\/ $\Complex$,
\[\sum_{i=1}^n\big|\tau_\Ic\big(E_{H}(\delta_i)V\big)\big|\leq \tau_\Ic\big(|\re(V)|\big)+\tau_\Ic\big(|\im(V)|\big).\]
\end{lemma}

\begin{proof}
Assume first that $V=V^*$.
Decompose $V=V_+-V_-$, with $0\leq V_+,V_-\in\Ic$.
Then,
\[\sum_{i=1}^n\big|\tau_\Ic\big(E_{H}(\delta_i)V_{\pm}\big)\big|=\tau_\Ic\big(E_{H}(\cup_{i=1}^n\delta_i)V_{\pm}\big)=
\tau_\Ic\big(V_\pm\big).\] Hence,
\[\sum_{i=1}^n\big|\tau_\Ic\big(E_{H}(\delta_i)V\big)\big|\leq\tau_\Ic(V_+)+\tau_\Ic(V_-)=\tau_\Ic(|V|).\]
If $V$ is not self-adjoint, then we decompose $V=\re(V)+\ii\,\im(V)$ and apply  the
just established estimate to $\re(V)$ and $\im(V)$.
\end{proof}

\begin{thm}\label{l4} Assume Hypothesis \ref{ideal}.
Let $\Omega$, $H_0$, $V$, and $\Fc$ satisfy Hypotheses \ref{cases}.
Then for all $f\in\Fc$ and all values of $t\in[0,1]$,
\begin{align}
\label{de}
\nonumber
&\left|\tau_\Ic\left(\frac{d}{dt}f(H_0+tV)\right)\right|\\
&\quad\leq \|f'\|_{L^\infty(\Omega)}\cdot \min\big\{\big(\tau_\Ic\big(|\re(V)|\big)+\tau_\Ic\big(|\im(V)|\big)\big),\,\|\tau_\Ic\|_{\Ic^*}\cdot\|V\|_\Ic\big\}.
\end{align}
\end{thm}
\begin{proof}
Assume that Hypotheses \ref{cases}\eqref{i} are satisfied.
Without loss of generality we assume $f\in C_c^2(\Reals)$; hence, $f\in W_2$.
By $\|\cdot\|_\Ic$--boundedness of the trace $\tau_\Ic$, Lemma \ref{l1a}\eqref{l1ai}, and cyclicity of traces
\begin{align*}
\tau_\Ic\left(\frac{d}{dt} f(H_0+tV)\right)
&=\frac{\ii}{\sqrt{2\pi}}\int_\Reals\int_0^\la \tau_\Ic\big(e^{\ii(\la-x)(H_0+tV)}Ve^{\ii x (H_0+tV)}\big)\hat f(\la)\,dx\,d\la\\
&=\frac{\ii}{\sqrt{2\pi}}\int_\Reals\tau_\Ic\big(e^{\ii\la(H_0+tV)}V\big)\la\hat f(\la)\,d\la\\
&=\tau_\Ic\left(\frac{\ii}{\sqrt{2\pi}}\int_\Reals e^{\ii\la(H_0+tV)}\la\hat f(\la)\,d\la\, V\right)\\
&=\tau_\Ic\big(f'(H_0+tV)V\big).
\end{align*}
For later use, we note that from the above representation, we immediately have the bound
\begin{equation*}
\left|\tau_\Ic\left(\frac{d}{dt}f(H_0+tV)\right)\right|\leq \|f'\|_{L^\infty(\Omega)}\cdot\|\tau_\Ic\|_{\Ic^*}\cdot\|V\|_\Ic.
\end{equation*}
By the spectral theorem, $f'(H_0+tV)$ can be approximated by a sequence of finite sums $\left\{\sum_{j=1}^n f'\big(\la_j^{(n)}\big)E_{H_0+tV}\big(\delta_j^{(n)}\big)\right\}_{n=1}^\infty$ in the uniform operator topology.
Hence, we have
\begin{align*}
\tau_\Ic\big(f'(H_0+tV)V\big)=\lim_{n\rightarrow\infty}\sum_{j=1}^n f'\big(\la_j^{(n)}\big)\,\tau_\Ic\big(E_{H_0+tV}\big(\delta_j^{(n)}\big)V\big).
\end{align*}
Finally, Lemma \ref{l3} ensures the estimate \eqref{de}.

Assume Hypotheses \ref{cases}\eqref{iii}.
We will prove the estimate \eqref{de} for every derivative $\left|\tau_\Ic\left(\frac{d}{dt}\big|_{t=t_0}f(H_0+tV)\right)\right|$,
with $t_0\in [0,1]$.
The case of $H_0$ unitary and $t_0=0$ can be handled in a straightforward manner
using Lemma \ref{l1a}\eqref{l1aii}.
The case of $H_0$ and $H_0+V$ contractions then follows using unitary dilations.
Indeed, from Proposition~\ref{prop:Udil}, we have a unitary dilation
$U_{t_0}$ of $H_0+t_0V$ in some $\sigma$--finite, semifinite von Neumann algebra factor $\Nc$ with
a trace--preserving identification of $\Bc$ with $p\Nc p$ for a projection $p\in\Nc$.
Consider the ideal $\Ict\subseteq\Nc$ with ideal norm $\|\cdot\|_\Ict$
and trace $\tau_\Ict$ from Proposition~\ref{prop:extend}.
Then
\[
g(H_0+tV)=p\,g(U_t)p
\]
for every polynomial $g$.
Hence, using Lemma \ref{l1a}\eqref{l1aii} twice, we get
\begin{align*}
\frac{d}{dt}\bigg|_{t=t_0}f(H_0+tV)
&=\sum_{k=1}^\infty\hat f(k)\sum_{\substack{0\leq k_0,k_1\\k_0+k_1=k-1}}(H_0+t_0V)^{k_0}V(H_0+t_0V)^{k_1} \\
&=\sum_{k=1}^\infty\hat f(k)\sum_{\substack{0\leq k_0,k_1\\k_0+k_1=k-1}}p(U_{t_0})^{k_0}pVp(U_{t_0})^{k_1}p \\
&=p\left(\frac{d}{dt}\bigg|_{t=0}f(U_{t_0}+tV)\right)p,
\end{align*}
where $V=pVp\in\Bc=p\Nc p$.
By the estimate \eqref{de} for a unitary operator,
\begin{align*}
\left|\tau_\Ic\left(\frac{d}{dt}\bigg|_{t=t_0}f(H_0+tV)\right)\right|
&=\left|\tau_\Ict\left(p\left(\frac{d}{dt}\bigg|_{t=0}f(U_{t_0}+tV)\right)p\right)\right| \\
&\leq\left|\tau_\Ict\left(\frac{d}{dt}\bigg|_{t=0}f(U_{t_0}+tV)\right)\right| \\
&\leq \|f'\|_{L^\infty(\Omega)}\cdot\big(\tau_\Ic\big(|\re(V)|\big)+\tau_\Ic\big(|\im(V)|\big)\big).
\end{align*}

Now we assume that Hypotheses \ref{cases}\eqref{ii} are satisfied.
Due to the linearity of $\tau_\Ic$ and the linearity of the G\^{a}teaux derivative of the operator function,
it is enough to prove the lemma for $f(\la)=(z-\la)^{-k}$, $k\in\Nats$, $\im(z)<0$.
Let $L_{t_0}$ be the self--adjoint dilation of the dissipative operator $H_0+t_0V$
as constructed in Proposition~\ref{prop:Adil},
with projection $p$ such that $p\,f'(L_{t_0})p=f'(H_0+t_0V)$ and $p\Nc p=\Bc$.
Again, let $\Ict$, $\|\cdot\|_\Ict$ and $\tau_\Ict$ be as in Proposition~\ref{prop:extend}.
By Lemma \ref{l1}\eqref{l1i} and the cyclicity of traces,
\begin{align*}
\left|\tau_\Ic\left(\frac{d}{dt}\bigg|_{t=t_0}f(H_0+tV)\right)\right|
&=\big|\tau_\Ic(f'(H_0+t_0V)V)\big| \\
&=\big|\tau_{\widetilde\Ic}(pf'(L_{t_0})pVp)\big|
\leq\big|\tau_\Ict(f'(L_{t_0})V)\big|.
\end{align*}
As in the above proof in the case that Hypotheses~\ref{cases}\eqref{i} hold,
since $f'$ is bounded and $L_{t_0}$ is self--adjoint and using Lemma~\ref{l3}, we conclude that \eqref{de} holds.
\end{proof}

\begin{remark}
Unitary dilations were applied in \cite{PSS-circle} to derive estimates for standard traces (defined on trace class elements of $\BH$) of higher order G\^{a}teaux derivatives of polynomials of contractions.
\end{remark}

Our first main theorem is proved below.

\begin{proof}[Proof of Theorem \ref{mt}]
One can see by the argument (based on the representations for operator derivatives from Lemmas \ref{l1} and \ref{l1a} and the resolvent identity) used in Lemma \ref{l6}  that the function $t\mapsto\frac{d}{dt}\,\tau_\Ic\big(f(H_0+tV)-f(H_0)\big)$ is continuous.
Along with Lemma \ref{l6}, this implies
\begin{align*}
\tau_\Ic\big(f(H_0+V)-f(H_0)\big)&=\int_0^1\frac{d}{dt}\,\tau_\Ic\big(f(H_0+tV)-f(H_0)\big)\,dt\\
&=\int_0^1\tau_\Ic\left(\frac{d}{dt}f(H_0+tV)\right)\,dt.
\end{align*}
Hence, by Theorem \ref{l4}, we have
\begin{align}
\label{te}
\big|\tau_\Ic\big(f(H_0+V)-f(H_0)\big)\big|\leq\|f'\|_{L^\infty(\Omega)}\cdot
\big(\tau_\Ic\big(|\re(V)|\big)+\tau_\Ic\big(|\im(V)|\big)\big).
\end{align}

Now the function
\begin{equation}\label{eq:fctl}
f'\mapsto\tau_\Ic\big(f(H_0+V)-f(H_0)\big)
\end{equation}
is well defined on the set of derivatives of the allowable functions in the various cases of Hypotheses~\ref{cases},
and we have just seen in~\eqref{te} that it is bounded with respect to the supremum norm in $C_0(\Omega)$.
By extending continuously to the closure and employing the Hahn--Banach theorem,
if necessary, we get a bounded linear functional on $C_0(\Omega)$ that extends the map~\eqref{eq:fctl}.
By the Riesz representation theorem, this map arises as integration against a complex, finite Borel measure $\mu_{H_0,V}$, with the bound on total variation as desired.
\end{proof}

\begin{remark}
\label{nr}
This paper does not aim to find the most general sets of functions $\Fc$ for which the trace formula \eqref{tf2} holds.  For instance, if $H_0=H_0^*$ and $V=V^*$, then in case of $H_0$ bounded, \eqref{tf2} also holds for $f\in C^2(\Reals)$ such that $f,f',f''$ are Fourier-Stieltjes transforms of finite measures and, in case of $H_0$ unbounded, for $f$ representable in the form $f(\la)=\int_\Pi\frac{1}{\la-z}\,\varpi(dz)$, where $\Pi\subseteq\Complex\setminus\Reals$ and the measure $\varpi$ satisfies $\int_\Pi\frac{1}{|\im(z)|^k}\,|\varpi|(dz)<\infty$, for $k=1,2$.
\end{remark}

In case of $H_0$ and $V$ unitary operators, considering multiplicative perturbations allows to extend \eqref{tf2} to the functions $f$ which along with $f'$ and $f''$ are given by absolutely convergent Fourier series.

\begin{thm}\label{thm:unitaries}
\label{ui}
Let $V\in\Ic$ and assume that $H_0\in\Bc$ and $H_0+V$ are unitary.
Let $\Fc$ be the set of all functions $f$ on the unit circle such that $f(z)=\sum_{k=-\infty}^\infty\hat{f}(k)z^k$,
for all $z\in\mathbb{T}$, and $\sum_{k=-\infty}^\infty|k(k-1)\hat{f}(k)|<\infty$.
Then, there is a measure $\mu_{H_0,V}$ on the unit circle such that the trace formula \eqref{tf2} holds
for all $f\in\Fc$, and its total variation is bounded by $\|\mu_{H_0,V}\|\leq\frac\pi2\,\|V\|_\Ic$.
\end{thm}

\begin{proof}
One can represent the unitary $(H_0+V)H_0^{-1}$ as $e^{\ii T}$, where $T=T^*$ and the spectrum of $T$ is contained in $(-\pi,\pi]$. From the inequality $\frac{2}{\pi}|x|\leq |e^{\ii x}-1|$, for $x\in(-\pi,\pi]$, one has by the spectral theorem that
\[|T|\leq\frac\pi2\big|e^{\ii T}-I\big|=\frac\pi2\big|(H_0+V)H_0^{-1}-I\big|=\frac\pi2\big|VH_0^{-1}\big|.\]
Hence, 
$T\in\Ic$ and
\begin{align}
\label{bT}
\|T\|_\Ic \leq\frac\pi2\,\|V\|_\Ic.
\end{align}
Consider the path of unitaries $t\mapsto U_t=e^{\ii Tt}H_0$ joining $H_0$ and $H_0+V$.
Since $\frac d{dt}U_t=iTU_t$ and $\frac d{dt}U_t^{-1}=U_t^{-1}(-iT)$, employing Lemma \ref{l1} \eqref{l1ii}, we obtain
\begin{align*}
\frac{d}{dt}f(U_t)=
\sum_{k=1}^\infty\hat{f}(k) \sum_{\substack{0\leq k_0,k_1\\k_0+k_1=k-1}}U_t^{k_0}\,\ii T\,U_t^{k_1+1}
+\sum_{k=1}^\infty-\hat{f}(-k)\sum_{\substack{0\leq k_0,k_1\\k_0+k_1=k-1}}U_t^{-k_0-1}\,\ii T\,U_t^{-k_1},
\end{align*}
where the series converge in $\|\cdot\|_\Ic$.
Further, by cyclicity of the trace, we get
\begin{align*}
\tau_\Ic\left(\frac{d}{dt}f(U_t)\right)&=\ii\sum_{k=1}^\infty k\hat{f}(k)\, \tau_\Ic\big(U_t^k\,T\big)+\ii\sum_{k=1}^\infty-k\hat{f}(-k)\tau_\Ic\big(U_t^{-k}\,T\big)\\
&=\tau_\Ic\left(\ii\sum_{k=-\infty}^\infty k\hat{f}(k)U_t^{k-1}\,U_tT\right)
=\ii\,\tau_\Ic\big(f'(U_t)U_tT\big).
\end{align*}
Hence, by 
the estimate \eqref{bT},
\begin{align}
\label{dem}
\bigg|\tau_\Ic\left(\frac{d}{dt}f(U_t)\right)\bigg|
\leq\frac\pi2\,\|f'\|_{L^\infty(\mathbb{T})}\cdot\|\tau_\Ic\|_{\Ic^*}\cdot\|V\|_\Ic.
\end{align}
Note that by Lemma \ref{l1} \eqref{l1ii},
\begin{align*}
f'(U_t)U_t-f'(U_{t_0})U_{t_0}&=\sum_{k=-\infty}^\infty k\hat{f}(k)\big(U_t^k-U_{t_0}^k\big)
=\sum_{k=1}^\infty k\hat{f}(k)\sum_{\substack{0\leq k_0,k_1\\k_0+k_1=k-1}}U_t^{k_0}\,(U_t-U_{t_0})\,U_t^{k_1}\\
&\quad+\sum_{k=1}^\infty -k\hat{f}(-k)\sum_{\substack{0\leq k_0,k_1\\k_0+k_1=k-1}}U_t^{-k_0}\,\big(U_t^{-1}-U_{t_0}^{-1}\big)\,U_t^{-k_1}.
\end{align*}
By Duhamel's formula,
\[\|U_t-U_{t_0}\|\leq |t-t_0|\cdot\|T\|,\quad \big\|U_t^{-1}-U_{t_0}^{-1}\big\|\leq |t-t_0|\cdot\|T\|.\]
Therefore,
\[\|f'(U_t)U_t-f'(U_{t_0})U_{t_0}\|\leq |t-t_0|\cdot\|T\|\cdot\sum_{k=-\infty}^\infty |k(k-1)\hat{f}(k)|,\]
implying that the function $t\mapsto f'(U_t)U_tT$ is uniformly continuous in $\|\cdot\|_\Ic$.
Hence, the function
\[t\mapsto \tau_\Ic\left(\frac{d}{dt}f(U_t)\right)=\tau_\Ic\big(f'(U_t)U_tT\big)\]
is uniformly continuous.
One can verify that
\[\frac{d}{dt}\tau_\Ic\big(f(U_t)-f(U_0)\big)=\tau_\Ic\left(\frac{d}{dt}f(U_t)\right).\]
Hence, by the fundamental theorem of calculus,
\begin{align*}
\tau_\Ic\big(f(H_0+V)-f(H_0)\big)=\int_0^1\frac{d}{dt}\tau_\Ic\big(f(U_t)-f(U_0)\big)\,dt
=\int_0^1\tau_\Ic\left(\frac{d}{dt}f(U_t)\right)\,dt,
\end{align*}
which along with \eqref{dem}, the Riesz representation and Hahn-Banach theorems completes the proof.
\end{proof}

We now prove the linearization formula~\eqref{lin},  and also its analogue
in a more general context of $\sigma$--finite, semifinite von Neumann algebra factors.

\begin{thm}
\label{imt}
Assume Hypotheses \ref{ideal} and \ref{cases} and assume $\tau_\Ic(\Ic^2)=\{0\}$.
Then
\begin{equation}\label{lintau}
\tau_\Ic\big(f(H_0+V)-f(H_0)\big)=\tau_\Ic\big(f'(H_0)V\big).
\end{equation}
\end{thm}

\begin{proof}
If Hypotheses~\ref{cases}\eqref{i} holds,
then without loss of generality, we can assume that $f$ is compactly supported and, hence, $f\in W_2$.
Then, by Lemma \ref{l1a}, Duhamel's formula, and the Lebesgue dominated convergence theorem for Bochner integrals,
\begin{align*}
&\tau_\Ic\left(f(H_0+V)-f(H_0)-\frac{d}{dt}\bigg|_{t=0}f(H_0+tV)\right)\\
&\quad=-\frac{1}{\sqrt{2\pi}}\tau_\Ic\left(\int_\Reals\int_0^{x_0}\int_0^{x_0-x_1}e^{\ii(x_0-x_1-x_2)(H_0+V)}Ve^{\ii x_2 H_0}V e^{\ii x_1 H_0}\hat f(x_0)\,dx_2\,dx_1\,dx_0\right)\\
&\quad=-\frac{1}{\sqrt{2\pi}}\int_\Reals\int_0^{x_0}\int_0^{x_0-x_1}\tau_\Ic\big(e^{\ii(x_0-x_1-x_2)(H_0+V)}Ve^{\ii x_2 H_0}V e^{\ii x_1 H_0}\big)\hat f(x_0)\,dx_2\,dx_1\,dx_0.
\end{align*}
By assumption, the latter integral equals zero.
It was established in the course of the proof of Theorem \ref{l4} that
\begin{align*}
\tau_\Ic\left(\frac{d}{dt}f(H_0+tV)\right)=\tau_\Ic\big(f'(H_0+tV)V\big).
\end{align*}
Therefore, we have
\[\tau_\Ic\big(f(H_0+V)-f(H_0)\big)=\tau_\Ic\big(f'(H_0)V\big).\]

The proof of \eqref{lintau} under Hypotheses \ref{cases}\eqref{ii} or \eqref{iii} is even simpler. We have that the crucial property $\left(f(H_0+V)-f(H_0)-\frac{d}{dt}\big|_{t=0}f(H_0+tV)\right)\in\Ic^2$ immediately follows from Lemma \ref{l1} \eqref{l1i} and Lemma \ref{l1a} \eqref{l1aii}.
\end{proof}

\section{Properties of spectral shift measures}
\label{notac}

The goal of this section is to establish properties of the spectral shift measures that are
distinct from those that we have in the case of normal traces.

For the next three propositions,
suppose $\Ic$ is a normed ideal with ideal norm $\|\cdot\|_\Ic$ and a positive, $\|\cdot\|_\Ic$--bounded
trace $\tau_\Ic$.

\begin{prop}
\label{triv}
Assume Hypotheses \ref{ideal} and \ref{cases}.
If $\tau_\Ic\big(|\re(V)|\big)+\tau_\Ic\big(|\im(V)|\big)=0$, then
Theorem~\ref{mt} holds with
$\mu_{H_0,V}=0$.
\end{prop}
\begin{proof}
The result immediately follows from the estimate \eqref{tfe} for the total variation of the measure $\mu_{H_0,V}$.
\end{proof}

In particular, for $\Bc=\BH$, $\Ic=\Lc^{(1,\infty)}$ and $\tau_\Ic=\Tr_\omega$, we have that the spectral shift measure
$\mu_{H_0,V}$ vanishes when its counterpart for the standard trace is defined, namely, when $V$ is trace class.


\begin{prop} Assume Hypotheses \ref{ideal}.
\label{dm}
Let $H_0=aI$, for $a\in\Reals$, and let $V=V^*\in\Ic$.
Under Hypotheses \ref{cases}(i) and assuming $\tau_\Ic(\Ic^2)=\{0\}$,
then $\mu_{H_0,V}=\tau_\Ic(V)\,\delta_a$.
\end{prop}
\begin{proof}
Using direct calculation and $\tau_\Ic(\Ic^2)=\{0\}$, we get
\begin{align*}
\tau_\Ic\big((aI+V)^k-a^kI\big)=\tau_\Ic\left(\sum_{j=1}^k\begin{pmatrix}k\\j\end{pmatrix}a^{k-j}V^j\right)
=ka^{k-1}\tau_\Ic(V).
\end{align*}
So by Theorem \ref{mt},
\begin{align*}
\tau_\Ic\big((aI+V)^k-a^kI\big)=\int_\Reals k\la^{k-1}\,\mu_{H_0,V}(dt),
\end{align*}
for any $k\in\Nats$, the denseness of the polynomials in the space of real-valued continuous functions on a compact ($\mu_{H_0,V}$ is compactly supported because $H_0$ is bounded) implies the result.
\end{proof}

The next proposition gives a sufficient condition for absolute continuity (with respect to the Lebesgue measure) of the spectral shift measures for pairs of contractions.

\begin{prop}\label{amc}  Assume Hypotheses \ref{ideal}.
Under Hypotheses \ref{cases}(iii) and assuming $\tau_\Ic(\Ic^2)=\{0\}$,
if $H_0,V\in\Ic$,
then Theorem~\ref{mt} holds
with the measure $\mu_{H_0,V}$ absolutely continuous and, in fact,
a constant multiple of Haar measure on the unit circle.
\end{prop}
\begin{proof}
Since $\tau_\Ic(H_0^2)=0$ and $\tau_\Ic(H_0V)=0$, we derive from Lemma \ref{l1} \eqref{l1ii} that for $f$ a polynomial,
\begin{align*}
\tau_\Ic\big(f(H_0+V)-f(H_0)\big)
=\begin{cases}a\,\tau_\Ic(V) &\text{ if } f(z)=az\\ 0 &\text{ if } f\in\text{\rm span}\{z^2,z^3,z^4,\ldots\}.\end{cases}
\end{align*}
Comparison with the trace formula \eqref{tf2} gives
\begin{align*}
\int_\mathbb{T}z^n\,d\mu_{H_0,V}(z)
=\begin{cases}\tau_\Ic(V) &\text{ if } n=0\\ 0 &\text{ if } n\in\Nats.\end{cases}
\end{align*}
\end{proof}

For the remainder of this section, we focus on the case
when $\Bc=\BH$, $\Ic=\Lc^{(1,\infty)}$ and $\tau_\Ic=\Tr_\omega$,
and we show that any finite positive measure supported in a compact subset of $\Reals$
is the spectral shift measure for a pair of commuting self-adjoint operators.

\begin{thm}\label{prop:order1example}
Let $\sigma$ be any finite positive measure having bounded support in the real line.
Then there are (commuting) diagonal operators $H_0=H_0^*\in\BH$ and $V=V_0^*\in\Lc^{(1,\infty)}$ such that,
under Hypotheses \ref{cases}(i), we have $\mu_{H_0,V}=\sigma$ in Theorem~\ref{mt}.
\end{thm}

In the proof, we will use the following easy result.
\begin{lemma}\label{lem:mut}
Let $M>0$, $p\in\Nats$ and $\eps>0$.
Let $a_1,\ldots,a_p\in[-M,M]$ and
consider the measures
\[
\mu=\frac1p\sum_{k=1}^p\delta_{a_k},\qquad
\mut=\frac1p\sum_{k=1}^pw_k\delta_{a_k}
\]
where $w_k\in(1-\eps,1+\eps)$ for all $k$ and $\sum_{k=1}^pw_k=p$.
Then for all $f\in C([-M,M])$ we have
\[
\biggl|\int f\,d\mu-\int f\,d\mut\,\biggr|\le\eps\|f\|_\infty\,,
\]
where $\|\cdot\|_\infty$ is the supremum norm on $C([-M,M])$.
\end{lemma}
\begin{proof}
\[
\biggl|\int f\,d\mu-\int f\,d\mut\,\biggr|=
\biggl|\frac1p\sum_{k=1}^pf(a_k)(1-w_k)\biggr|\le
\frac1p\sum_{k=1}^p|f(a_k)|\,|1-w_k|
\le\eps\|f\|_\infty\,.
\]
\end{proof}

\begin{proof}[Proof of Theorem~\ref{prop:order1example}]
Without loss of generality suppose $\sigma$ is a probability measure.
Let $M>0$ be such that the support of $\sigma$ lies in $[-M,M]$.
The basic idea is simple:  to write $H_0$ as a direct sum of diagonal blocks 
whose spectral measures approximate
better and better $\sigma$, and so that the blocks are small enough that the
variation caused by the weight $\frac1n$ from $V$ makes only small distortions.
In particular, using standard approximation techniques
we can find positive integers $p(1),p(2),\ldots$ and
\[
a^{(i)}_1,a^{(i)}_2,\ldots,a^{(i)}_{p(i)}\in[-M,M]
\]
such that, letting $q(i)=p(1)+p(2)+\cdots+p(i)$, we have
\begin{equation}\label{eq:piqi}
\lim_{i\to\infty}\frac{p(i+1)}{q(i)}=0
\end{equation}
and letting
\[
\mu_i=\frac1{p(i)}\sum_{k=1}^{p(i)}\delta_{a^{(i)}_k},
\]
the sequence $(\mu_i)_{i=1}^\infty$ of measures converges in weak$^*$--topology on $C([-M,M])^*$ to $\sigma$.
Indeed, to find $a^{(i)}_j$ so that the measures $\mu_i$ converge as required without requiring~\eqref{eq:piqi} to hold
is a standard discretization argument, and ensuring~\eqref{eq:piqi} holds can be accomplished by sufficient repetition
of the blocks $a^{(i)}_1,a^{(i)}_2,\ldots,a^{(i)}_{p(i)}$, if necessary.
Let
\[
A^{(i)}=\diag(a^{(i)}_1,\ldots,a^{(i)}_{p(i)})
\]
and consider the diagonal bounded operator
$H_0=A^{(1)}\oplus A^{(2)}\oplus\cdots\in\BH$.
Let $V=\diag((\frac1n)_{n=1}^\infty)$.
Then $V\in\Ic$.
We will show $\mu_{H_0,V}=\sigma$.

For ease of calculation, we will alter the formula for $\Tr_\omega$ by replacing $\log(n+1)$ in the denominator
of~\eqref{eq:trB} by $1+\frac12+\frac13+\cdots+\frac1n$.
Clearly, this alters the value of $\Tr_\omega$ only by a strictly positive multiplicative constant.

Since $\Tr_\omega$ vanishes on $\Ic^2$,
we have
\begin{align*}
\int(k\lambda^{k-1})\,d\mu_{H_0,V}(\lambda)&=
\Tr_\omega((H_0+V)^k-H_0^k)=k\Tr_\omega(H_0^{k-1}V)\\
&=k\lim_{N\to\omega}\frac1{1+\frac12+\cdots+\frac1N}\sum_{j=1}^N\frac1j\,b_j^{k-1},
\end{align*}
where
\[
(b_1,b_2,\ldots)=(a^{(1)}_1,\ldots,a^{(1)}_{p(1)},a^{(2)}_1,\ldots,a^{(2)}_{p(2)},\ldots).
\]
Since the supports of $\mu_{H_0,V}$ and $\sigma$ are bounded, it will suffice to show
\[
\int\lambda^{k-1}\,d\mu_{H_0,V}(\lambda)=\int\lambda^{k-1}\,d\sigma(\lambda)
\]
for all $k\in\Nats$,
and we will actually prove the stronger statement
\begin{equation}\label{eq:limsig}
\lim_{N\to\infty}\frac1{1+\frac12+\cdots+\frac1N}\sum_{j=1}^N\frac1j\,b_j^{k-1}
=\int\lambda^{k-1}\,d\sigma(\lambda).
\end{equation}
Let
\[
s(i)=\sum_{j=1}^{p(i)}\frac1{q(i-1)+j}\,.
\]
Let $N\in\Nats$ and let $l\ge1$ be such that
$q(l)<N\le q(l+1)$.
Then
\[
1+\frac12+\cdots+\frac1N=s(1)+s(2)+\cdots+s(l)+e,
\]
where
$e=\sum_{j=1}^{N-q(l)}\frac1{q(l)+j}$.
We have
\begin{align*}
&\frac1{1+\frac12+\cdots+\frac1N}\sum_{j=1}^N\frac1j\,b_j^{k-1} \\
&\quad=\frac1{s(1)+\cdots+s(l)+e}\left(
\sum_{i=1}^l\sum_{j=1}^{p(i)}\frac1{q(i-1)+j}(a^{(i)}_j)^{k-1}
+\sum_{j=1}^{N-q(l)}\frac1{q(l)+j}(a_j^{(l+1)})^{k-1}\right) \\
&\quad=\frac1{s(1)+\cdots+s(l)+e}\left(
\sum_{i=1}^ls(i)\int\lambda^{k-1}\,d\mut_i(\lambda)
+e\int\lambda^{k-1}\,d\eta_N(\lambda)\right),
\end{align*}
for probability measures
\begin{align*}
\mut_i&=\frac1{s(i)}\sum_{j=1}^{p(i)}\left(\frac1{q(i-1)+j}\right)\delta_{a^{(i)}_j} \\
\eta_N& 
=\frac1e\sum_{j=1}^{N-q(l)}\left(\frac1{q(l)+j}\right)\delta_{a^{(l+1)}_j.}
\end{align*}
Thus,
\[
\frac1{1+\frac12+\cdots+\frac1N}\sum_{j=1}^N\frac1j\,b_j^{k-1}=\int\lambda^{k-1}\,d\rho_N(\lambda),
\]
where $\rho_N$ is the convex combination
\begin{equation}\label{eq:conv}
\rho_N=\sum_{i=1}^l\left(\frac{s(i)}{s(1)+\cdots+s(l)+e}\right)\mut_i+\left(\frac e{s(1)+\cdots+s(l)+e}\right)\eta_N\,.
\end{equation}
Also, we have $e\le p(l+1)/q(l)$, so by~\eqref{eq:piqi}, choosing $N$ sufficiently large ensures that $e$ is arbitrarily small,
and, thus, $e/(s(1)+\cdots+s(l)+e)$ is arbitrarily small.

Now let us examine the measures $\mut_i$.
Setting $q(0)=0$,
we have
\[
\mut_i=\frac1{p(i)}\sum_{j=1}^{p(i)}w^{(i)}_j\delta_{a^{(i)}_j}
\]
where
\[
w^{(i)}_j=\frac{p(i)}{s(i)(q(i-1)+j)}\,,
\]
so that $\sum_{j=1}^{p(i)}w^{(i)}_j=p(i)$
and $w^{(i)}_1\ge w^{(i)}_2\ge\cdots\ge w^{(i)}_{p(i)}>0$.
But
\[
1\le\frac{w^{(i)}_1}{w^{(i)}_{p(i)}}=\frac{q(i)}{q(i-1)+1}<\frac{q(i)}{q(i-1)}=1+\frac{p(i)}{q(i-1)}
\]
and by the condition~\eqref{eq:piqi}, the right--hand--side tends to $1$ as $i\to\infty$.
Since the average of the positive numbers $w^{(i)}_1,\ldots,w^{(i)}_{p(i)}$ equals $1$, we have $w^{(i)}_1\ge1\ge w^{(i)}_{p(i)}$;
since for $\delta_1\ge0$ and $1>\delta_2\ge0$ we have
\[
\frac{1+\delta_1}{1-\delta_2}-1=\frac{\delta_1+\delta_2}{1-\delta_2}\ge\max\{\delta_1,\delta_2\},
\]
from $\lim_{i\to\infty}w^{(i)}_1/w^{(i)}_{p(i)}=1$
we get
\[
\lim_{i\to\infty}\max_{1\le j\le p(i)}|1-w^{(i)}_j|=0.
\]
Consequently, applying Lemma~\ref{lem:mut},
the sequence $\mut_i$ of measures converges as $i\to\infty$ in weak$^*$--topology on $C([-M,M])$ to $\sigma$.
Now, since $s(1)+\cdots+s(l)+e$ diverges to $\infty$ as $N\to\infty$ and, as noted before, $e\to0$ as $N\to\infty$,
from~\eqref{eq:conv} we see that $\rho_N$
converges as $N\to\infty$ in weak$^*$--topology on $C([-M,M])$ to $\sigma$.
This yields~\eqref{eq:limsig}, as desired.
\end{proof}

\section{Second order spectral shift measures}
\label{sec:2ndOrd}

The goals of this section are to establish the trace formula \eqref{2tf}
and a more general version,
as well as some properties of the second order spectral shift measure.
The first goal will be accomplished in Theorem \ref{2mt}, which holds under the hypotheses below.
Again, we work with a normed ideal $\Ic$
of a $\sigma$--finite, semifinite von Neumann algebra factor $\Bc$,
with ideal norm denoted $\|\cdot\|_\Ic$, and
endowed with a trace $\tau_\Ic:\Ic\to\Complex$ that is positive and $\|\cdot\|_\Ic$--bounded,
but we also  assume the following.
\begin{hyp}
\label{root}
$\Ic^{1/2}$ is a normed ideal with ideal norm $\|\cdot\|_{\Ic^{1/2}}$ and
the inequality~\eqref{eq:I12} holds.
\end{hyp}
By Proposition~\ref{prop:Malpha}, many Marcinkiewicz ideals, on which Dixmier traces are defined, satisfy Hypotheses \ref{root}.


\begin{hyp}
\label{2cases}
Consider a set $\Omega$, a closed, densely defined operator $H_0$ affiliated with $\Bc$,
$V\in\Ic^{1/2}$ and a set $\Fc$ of functions
that satisfy one of the following assertions:
\begin{enumerate}[(i)]
\item \label{2ii} $\Omega=\Reals$, $H_0$ and $H_0+V$ are dissipative, and
\[
\Fc=\text{\rm span}\left\{\la\mapsto (z-\la)^{-k}:\, k\in\Nats,\, \im(z)<0\right\};
\]
\item \label{2iii} $\Omega=\mathbb{T}$, $\|H_0\|\leq 1$, $\|H_0+V\|\leq 1$, and
$\Fc$ is the set of all functions that are
analytic on discs centered at $0$ and of radius strictly larger than $1$.
\end{enumerate}
\end{hyp}

\begin{thm}
\label{2mt} Assume Hypotheses \ref{ideal} and \ref{root}.
Let $\Omega$, $H_0$, $V$ and $\Fc$ satisfy Hypotheses \ref{2cases}.
Then, there exists a (countably additive, complex)
measure $\nu_{H_0,V}$ on $\Omega$ such that for every $f\in\Fc$, the trace formula
\begin{align}
\label{2tftau}
\tau_\Ic\left(f(H_0+V)-f(H_0)-\frac{d}{dt}\bigg|_{t=0}f(H_0+tV)\right)=\int_\Omega f''(\la)\,\nu_{H_0,V}(d\la)
\end{align}
holds.
Moreover, the total variation of $\nu_{H_0,V}$
is bounded as follows:
\begin{align}
\label{2tfe}
\big\|\nu_{H_0,V}\big\|\leq \frac12\,\tau_\Ic(|V|^2).
\end{align}
\end{thm}

The proof is based on the set of lemmas below.

By evaluating the derivatives of expressions in Lemmas \ref{l1} and \ref{l1a} \eqref{l1aii}, we obtain the following.

\begin{lemma}\label{2l1} Assume Hypotheses \ref{ideal} and \ref{root}.
\begin{enumerate}[(i)]
\item \label{2l1i}
Let $H_0$ be
be affiliated with $\Bc$ and $V\in\Bc$; let $H_t:=H_0+tV$.
Then, for $z\in\Complex$ such that $\sup\limits_{t\in [0,1]}\|(zI-H_t)^{-1}\|<\infty$ and for $k\in\Nats$,  $t_0\in [0,1]$,
\begin{align*}
\frac{d^2}{dt^2}\bigg|_{t=t_0}\big((zI-H_t)^{-k}\big)=2\sum_{\substack{1\leq k_0,k_1,k_2\leq k\\k_0+k_1+k_2=k+2}}(zI-H_{t_0})^{-k_0}V(zI-H_{t_0})^{-k_1}V(zI-H_{t_0})^{-k_2}.
\end{align*}
\item \label{2l1ii}
Let $H_0,V\in\Bc$. Then, for $k\in\Nats$, $t_0\in [0,1]$,
\begin{align*}
\frac{d^2}{dt^2}\bigg|_{t=t_0}\big(H_t^k\big)=2\sum_{\substack{0\leq k_0,k_1,k_2\\k_0+k_1+k_2=k-2}}H_{t_0}^{k_0}V H_{t_0}^{k_1}V H_{t_0}^{k_2}.
\end{align*}
\item \label{2l1iii}
Let $H_0\in\Bc$, $\|H_0\|\leq 1$, $V\in\Ic^{1/2}$, and $\|H_0+V\|\leq 1$. Then, for every $f$ analytic on a disc of radius $r>1$ centered at $0$ and $t_0\in [0,1]$,
\[
\frac{d^2}{dt^2}\bigg|_{t=t_0}f(H_t)
=2\sum_{k=2}^\infty \hat f(k) \sum_{\substack{0\leq k_0,k_1,k_2\\k_0+k_1+k_2=k-2}}H_{t_0}^{k_0}VH_{t_0}^{k_1}V H_{t_0}^{k_2},\]
where the series converges in $\|\cdot\|_\Ic$.
\end{enumerate}
\end{lemma}

Note that if $H_0$, $V$, and $\Fc$ satisfy Hypotheses \ref{2cases}, then for every $f\in\Fc$,
\[R_{H_0,V}(f):=f(H_0+V)-f(H_0)-\frac{d}{dt}\bigg|_{t=0}f(H_0+tV)\] is an element of $\Ic$.
This follows from Lemmas \ref{l1} and \ref{l1a}
and can be obtained similarly to \eqref{expl1} and \eqref{expl2} in the proof of the lemma below.

\begin{lemma}\label{lem:unifconts} Assume Hypotheses \ref{ideal} and \ref{root}.
If $H_0$, $V$, and $f\in\Fc$ are as in one of the cases of Hypotheses~\ref{2cases}, then
the function
\[
t\mapsto \frac{d}{ds}\bigg|_{s=t}f(H_0+sV)-\frac{d}{ds}\bigg|_{s=0}f(H_0+sV)
\]
is uniformly continuous and, hence, Bochner
integrable on $[0,1]$ with respect to $\|\cdot\|_\Ic$.
\end{lemma}
\begin{proof}
We will first demonstrate this in case $f(\la)=\la^k$, $k\in\Nats$. Applying Lemma \ref{l1} \eqref{l1ii} gives
\begin{align}
\label{expl1}
\nonumber
&\frac{d}{ds}\bigg|_{s=t}(H_0+sV)^k-\frac{d}{ds}\bigg|_{s=0}(H_0+sV)^k\\
&\quad=\sum_{\substack{0\leq k_0,k_1\\k_0+k_1=k-1}}(H_0+tV)^{k_0}V(H_0+tV)^{k_1}
-\sum_{\substack{0\leq k_0,k_1\\k_0+k_1=k-1}}H_0^{k_0}VH_0^{k_1},
\end{align}
which equals
\begin{align*}
&\sum_{\substack{0\leq k_0,k_1\\k_0+k_1=k-1}}(H_0+tV)^{k_0}V(H_0+tV)^{k_1}
-\sum_{\substack{0\leq k_0,k_1\\k_0+k_1=k-1}}H_0^{k_0}V(H_0+tV)^{k_1}\\
&\quad+\sum_{\substack{0\leq k_0,k_1\\k_0+k_1=k-1}}H_0^{k_0}V(H_0+tV)^{k_1}
-\sum_{\substack{0\leq k_0,k_1\\k_0+k_1=k-1}}H_0^{k_0}VH_0^{k_1}.
\end{align*}
By Lemma \ref{l1} \eqref{l1ii}, the latter equals
\begin{align}
\label{expl2}
\nonumber
&t\sum_{\substack{0\leq k_0,k_1\\k_0+k_1=k-1}}\;\sum_{\substack{0\leq i_0,i_1\\ii_0+i_1=k_0-1}}
(H_0+tV)^{i_0}VH_0^{i_1}V(H_0+tV)^{k_1}\\
&\quad+t\sum_{\substack{0\leq k_0,k_1\\k_0+k_1=k-1}}\;\sum_{\substack{0\leq j_0,j_1\\j_0+j_1=k_1-1}}
H_0^{k_0}V(H_0+tV)^{j_0}VH_0^{j_1}.
\end{align}
Since $t\mapsto H_0+tV$ is uniformly continuous on $[0,1]$ in the operator norm (this can be derived from Lemma \ref{l1} \eqref{l1ii}) and $V\in\Ic^{1/2}$, we obtain continuity of the function $t\mapsto \frac{d}{ds}\big|_{s=t}(H_0+sV)^k-\frac{d}{ds}\big|_{s=0}(H_0+sV)^k$ in $\|\cdot\|_{\Ic}$.
Now the uniform continuity for analytic functions as in Hypotheses~\ref{2cases}(ii) follows by norm estimates~\eqref{eq:I12}
and uniform convergence.

The case of functions $\lambda\mapsto(z-\lambda)^{-k}$ for $k\in\Nats$ and $\im(z)<0$,
can be proved similarly.
\end{proof}

\begin{lemma} Assume Hypotheses \ref{ideal} and \ref{root}.
\label{ir} If $H_0$, $V$, and $\Fc$ satisfy Hypotheses \ref{2cases}, then for every $f\in\Fc$,
\[
\tau_\Ic\big(R_{H_0,V}(f)\big)=\int_0^1 (1-t)\tau_\Ic\left(\frac{d^2}{dt^2} f(H_0+tV)\right)\,dt.
\]
\end{lemma}
\begin{proof}
Using continuity in the operator norm, which follows from~\eqref{eq:I12} and Lemma~\ref{lem:unifconts},
and the fundamental theorem of calculus,
one can verify that for every bounded linear functional $\phi$ on $\Bc$, we have
\[\phi\big(R_{H_0,V}(f)\big)=\phi\left(\int_0^1\left(\frac{d}{ds}f(H_0+sV)-\frac{d}{dt}\bigg|_{t=0} f(H_0+tV)\right)\,ds\right)\]
(where we pull $\phi$ through the derivative in $\frac d{ds}$ based on convergence in operator norm).
Thus, we have
\[R_{H_0,V}(f)=\int_0^1\left(\frac{d}{ds}f(H_0+sV)-\frac{d}{dt}\bigg|_{t=0} f(H_0+tV)\right)\, ds.\]
Using that Bochner integrability with respect to $\|\cdot\|_\Ic$, we have
\[\tau_\Ic\big(R_{H_0,V}(f)\big)=\int_0^1\tau_\Ic\left(\frac{d}{ds}f(H_0+sV)-\frac{d}{dt}\bigg|_{t=0} f(H_0+tV)\right)\,ds.\]
Integrating by parts in the latter integral, we arrive at
\[\tau_\Ic\big(R_{H_0,V}(f)\big)=\int_0^1 (1-s)\bigg(\frac{d}{ds}\tau_\Ic\left(\frac{d}{ds}f(H_0+sV)-\frac{d}{dt}\bigg|_{t=0} f(H_0+tV)\right)\bigg)\,ds.\]
With use of Lemmas \ref{l1} and \ref{2l1}, one can verify that
\[\frac{d}{ds}\tau_\Ic\left(\frac{d}{ds}f(H_0+sV)-\frac{d}{dt}\bigg|_{t=0} f(H_0+tV)\right)
=\tau_\Ic\left(\frac{d^2}{ds^2} f(H_0+sV)\right),\]
which completes the proof of the lemma.
\end{proof}

Similarly to the case of the first order trace formula, in order to establish the second order trace formula, we need to prove the following bound for the second order G\^{a}teaux derivative.

\begin{thm}
\label{2l4} Assume Hypotheses \ref{ideal} and \ref{root}.
Let $\Omega$, $H_0$, $V$, and $\Fc$ satisfy Hypotheses \ref{2cases}. Then for all $f\in\Fc$,
\begin{align}
\label{2de}
\left|\tau_\Ic\left(\frac{d^2}{dt^2}f(H_0+tV)\right)\right|\leq \|f''\|_{L^\infty(\Omega)}\cdot \tau_\Ic(|V|^2).
\end{align}
\end{thm}

The proof is based on the lemma below.

\begin{lemma}
\label{2l3} Assume Hypotheses \ref{ideal}.
Let $H$ be a
normal operator affiliated to $\Bc$ and $V\in\Ic^{1/2}$.
Then, for arbitrary Borel partitions $\{\delta_{i}\}_{i=1}^m$ and $\{\delta'_{i}\}_{i=1}^{m'}$ of $\Complex$,
\begin{align*}
\sum_{i_0,i_1}\big|\tau_\Ic\big(E_H(\delta_{i_0})VE_H(\delta'_{i_1})VE_H(\delta_{i_0})\big)\big|\leq \tau_\Ic(|V|^2).
\end{align*}
\end{lemma}

\begin{proof}
By the Cauchy-Schwarz inequality,
\begin{align*}
&\sum_{i_0,i_1}\big|\tau_\Ic\big(E_H(\delta_{i_0})VE_H(\delta'_{i_1})VE_H(\delta_{i_0})\big)\big|\\
&\quad\leq \sum_{i_0,i_1}\big(\tau_\Ic\big(|E_H(\delta_{i_0})VE_H(\delta'_{i_1})|^2\big)\big)^{1/2}
\big(\tau_\Ic\big(|E_H(\delta'_{i_1})VE_H(\delta_{i_0})|^2\big)\big)^{1/2}\\
&\quad\leq \left(\sum_{i_0,i_1}\tau_\Ic\big(|E_H(\delta_{i_0})VE_H(\delta'_{i_1})|^2\big)\right)^{1/2}
\left(\sum_{i_0,i_1}\tau_\Ic\big(|E_H(\delta'_{i_1})VE_H(\delta_{i_0})|^2\big)\right)^{1/2}\\
&\quad =\tau_\Ic(|V|^2).
\end{align*}
\end{proof}

The proof of Theorem \ref{2l4} will involve the divided difference of a function.
Recall that the divided difference of order $n$ is an operation on functions $f$ defined recursively as follows:
\begin{align*}
f^{[0]}(\la)&:=f(\la),\\
f^{[n]}(\la_0,\ldots,\la_n)&:=
\begin{cases}
\frac{f^{[n-1]}(\la_0,\ldots,\la_{n-2},\la_{n-1})-f^{[n-1]}(\la_0,\ldots,\la_{n-2},\la_n)}{\la_{n-1}-\la_n} &\text{ if } \la_{n-1}\neq\la_n\\
\frac{\partial}{\partial t}\big|_{t=\la_n}f^{[n-1]}(\la_0,\ldots,\la_{n-2},t) &\text{ if } \la_{n-1}=\la_n.
\end{cases}
\end{align*}
We 
have the following bound for the divided difference of $f\in C_b^2$: 
\begin{align}
\label{ddd}
\big\|f^{[2]}\big\|_\infty\leq\frac12\, \|f''\|_\infty.
\end{align}

Below we provide formulas for the second order divided differences of the functions involved in the proof of Theorem \ref{2l4}.

\begin{lemma}
\label{2ld}
The following assertions hold.
\begin{enumerate}[(i)]
\item \label{2ldiii}
For $f(\la)=(z-\la)^{-k}$, $k\in\Nats$, and $z,\la_0,\la_1,\la_2$ such that $f^{[2]}(\la_0,\la_1,\la_2)$ is well defined,
\[f^{[2]}(\la_0,\la_1,\la_2)=\sum_{\substack{1\leq k_0,k_1,k_2\leq k\\k_0+k_1+k_2=k+2}}(z-\la_0)^{-k_0}(z-\la_1)^{-k_1}(z-\la_2)^{-k_2}.\]

\item \label{2ldi}
For $f(\la)=\la^k$, $k\in\Nats$, and $\la_0,\la_1,\la_2\in \Complex$,
\[f^{[2]}(\la_0,\la_1,\la_2)=\sum_{\substack{0\leq k_0,k_1,k_2\\k_0+k_1+k_2=k-2}}\la_0^{k_0}\la_1^{k_1}\la_2^{k_2}.\]

\item \label{2ldii}
For $f$ analytic on a disc of radius $r>1$ centered at $0$ and $\la_0,\la_1,\la_2\in \overline{\mathbb{D}}$,
\[f^{[2]}(\la_0,\la_1,\la_2)=\sum_{k=1}^\infty \hat f(k)\sum_{\substack{0\leq k_0,k_1,k_2\\k_0+k_1+k_2=k-2}}\la_0^{k_0}\la_1^{k_1}\la_2^{k_2}.\]
\end{enumerate}
\end{lemma}

\begin{proof}[Proof of Theorem \ref{2l4}]
Note that, in case both $H_0$ and $H_0+V$ are self-adjoint or both are unitary,
the estimate \eqref{2de} would follow from the representation
\begin{align}
\label{dviadd}
\nonumber
&\tau_\Ic\left(\frac{d^2}{dt^2}f(H_0+tV)\right)\\
&\quad=2\lim_{n\rightarrow\infty}\sum_{i_0=1}^n\sum_{i_1=1}^n f^{[2]}\big(\la_{i_0}^{(n)},\la_{i_1}^{(n)},\la_{i_0}^{(n)}\big)
\,\tau_\Ic\left(E_{H_t}\big(\delta_{i_0}^{(n)}\big)V E_{H_t}\big(\delta_{i_1}^{(n)}\big)V\right),
\end{align}
for certain $\la_j^{(n)}\in\Omega$, $j=1,\ldots,n$,  and
partitions $\big(\delta_{i}^{(n)}\big)_{i=1}^n$ of $\Complex$,
in combination with
Lemma \ref{2l3}, and the bound \eqref{ddd}.

\medskip\noindent
{\em Case 1:} $H_0=H_0^*$ (possibly unbounded, affiliated to $\Bc$), $V=V^*\in\Ic^{1/2}$, and
\[f\in\text{\rm span}\left\{\la\mapsto (z-\la)^{-k}:\, k\in\Nats,\, \im(z)\neq 0\right\}.\]
In order to prove~\eqref{dviadd} for all such $f$,
it will enough to consider $f(\la)=(z-\la)^{-k}$.
By Lemma \ref{2l1} \eqref{2l1i},
\begin{align*}
&\tau_\Ic\left(\frac{d^2}{dt^2}f(H_0+tV)\right)=
\\&\quad=2\sum_{\substack{1\leq k_0,k_1,k_2\leq k\\k_0+k_1+k_2=k+2}}
\tau_\Ic\big((zI-H_t)^{-k_0}V(zI-H_t)^{-k_1}V(zI-H_t)^{-k_2}\big)\\
&\quad=2\sum_{\substack{1\leq k_0,k_1,k_2\leq k\\k_0+k_1+k_2=k+2}}
\tau_\Ic\big((zI-H_t)^{-k_0-k_2}V(zI-H_t)^{-k_1}V\big).
\end{align*}
By the spectral theorem,
there are Borel partitions $(\delta_i^{(n)})_{1\le i\le n}$ of $\Complex$ and complex numbers $\lambda_i^{(n)}$ such that
for any $m\in\{1,\ldots,k+2\}$ and any $t\in[0,1]$, the quantity
\[
\sum_{j=1}^n \big(z-\la_j^{(n)}\big)^{-m}E_{H_t}\big(\delta_j^{(n)}\big)
\]
converges in operator norm to $(zI-H_t)^{-m}$ as $n\to\infty$.
It follows, using the norm estimate~\eqref{eq:I12}, that for any $m_0,m_1\in\{1,\ldots,k+2\}$,
the quantity
\[
\sum_{i_0=1}^n\sum_{i_1=1}^n\big(z-\la_{i_0}^{(n)}\big)^{-m_0}\big(z-\la_{i_1}^{(n)}\big)^{-m_1}
 E_{H_t}\big(\delta_{i_0}^{(n)}\big)VE_{H_t}\big(\delta_{i_1}^{(n)}\big)V
\]
converges in $\|\cdot\|_\Ic$ to $(zI-H_t)^{-m_0}V(zI-H_t)^{-m_1}V$ as $n\to\infty$.
Let us write $\rho_t(A,B)=\tau_\Ic\big(E_{H_t}(A)V E_{H_t}(B)V\big)$.
Then,
\begin{align*}
&\tau_\Ic\left(\frac{d^2}{dt^2}f(H_0+tV)\right)=
\\&\quad=2\sum_{\substack{1\leq k_0,k_1,k_2\leq k\\k_0+k_1+k_2=k+2}}\lim_{n\rightarrow\infty}\,
\sum_{i_0=1}^n\sum_{i_1=1}^n\big(z-\la_{i_0}^{(n)}\big)^{-k_0-k_2}\big(z-\la_{i_1}^{(n)}\big)^{-k_1}
\,\rho_t\big(\delta_{i_0}^{(n)},\delta_{i_1}^{(n)}\big)
\\&\quad=2\lim_{n\rightarrow\infty}\sum_{i_0=1}^n\sum_{i_1=1}^n\sum_{\substack{1\leq k_0,k_1,k_2\leq k\\k_0+k_1+k_2=k+2}}
\big(z-\la_{i_0}^{(n)}\big)^{-k_0-k_2}\big(z-\la_{i_1}^{(n)}\big)^{-k_1}
\,\rho_t\big(\delta_{i_0}^{(n)},\delta_{i_1}^{(n)}\big).
\end{align*}
By Lemma \ref{2ld} \eqref{2ldiii}, the latter equals the limit in \eqref{dviadd}.

\medskip\noindent
{\em Case 2:} $V\in\Ic^{1/2}$, $H_0$ is unitary, $H_0+V$ is a contraction, and $f$ is a function analytic on a disc of radius $r>1$ centered at $0$.

Firstly, we make an additional assumption that $f$ is a polynomial.
The formula~\eqref{dviadd} and the estimate for
$\left|\tau_\Ic\left(\frac{d^2}{dt^2}\big|_{t=0} f(H_0+tV)\right)\right|$
can be derived completely analogously to the estimate in Case 1.
Now for $f$ analytic as above,
we approximate $f$ by a sequence of polynomials $f_n(z)=\sum_{k=0}^n \hat f(k)z^k$,
so that $\{f_n''\}_{n=1}^\infty$ converges to $f''$ in $L^\infty(\mathbb{T})$.
By Lemma \ref{2l1} \eqref{2l1iii}, we have
\begin{align*}
&\tau_\Ic\left(\frac{d^2}{dt^2}f(H_0+tV)\right)-\tau_\Ic\left(\frac{d^2}{dt^2}f_n(H_0+tV)\right)\\
&\quad=2\sum_{k=n+1}^\infty \hat f(k) \sum_{\substack{0\leq k_0,k_1,k_2\\k_0+k_1+k_2=k-2}}\tau_\Ic\big(H_t^{k_0}VH_t^{k_1}VH_t^{k_2}\big).
\end{align*}
Thus, by $\|\cdot\|_\Ic$--boundedness of $\tau_\Ic$ and the inequality~\eqref{eq:I12},
\begin{align*}
&\sup_{t\in[0,1]}\left|\tau_\Ic\left(\frac{d^2}{dt^2}f(H_0+tV)\right)
-\tau_\Ic\left(\frac{d^2}{dt^2}f_n(H_0+tV)\right)\right|\\
&\quad\leq 2\, \cdot\|\tau_\Ic\|_{\Ic^*}\cdot\big\||V|^2\big\|_\Ic\sum_{k=n+1}^\infty k(k-1)\,|\hat f(k)|,
\end{align*}
which converges to zero as $n\rightarrow\infty$.
Therefore,
\begin{align*}
\left|\tau_\Ic\left(\frac{d^2}{dt^2}\bigg|_{t=0}f(H_0+tV)\right)\right|
&=\lim_{n\rightarrow\infty}\left|\tau_\Ic\left(\frac{d^2}{dt^2}\bigg|_{t=0}f_n(H_0+tV)\right)\right|\\
&\leq\tau_\Ic\big(|V|^2\big)\lim_{n\rightarrow\infty}\|f_n\|_{L^\infty(\mathbb{T})}
=\tau_\Ic\big(|V|^2\big)\|f\|_{L^\infty(\mathbb{T})}.
\end{align*}

The estimate for the derivative at $t=t_0\neq 0$ can be derived by dilating the contraction $H_0+t_0V$ to a unitary operator, as it was done in the proof of Theorem \ref{l4}.

Finally, in case of dissipative (respectively, contractive) operators $H_0$ and $H_0+V$,
the estimate \eqref{2de} follows from the self-adjoint (respectively, unitary) case
and use of the self-adjoint (unitary) dilations results from Subsection~\ref{subsec:dilation}
and Proposition~\ref{prop:extend},
similarly to how it was done in the proof of Theorem \ref{l4}.

\end{proof}

\begin{proof}[Proof of Theorem \ref{2mt}]
The result follows upon applying Lemma \ref{ir}, Theorem \ref{2l4}, the Riesz representation theorem, and, in case of non-self-adjoint and non-unitary operators, the Hahn-Banach theorem.
\end{proof}

By adjusting the proof of Theorem \ref{imt}, we obtain the following generalization
of the formula~\eqref{2lin}
for the second order remainder of the Taylor approximation (again, with non--optimal set of functions $f$).

\begin{thm}
\label{2imt} Assume Hypotheses \ref{ideal} and \ref{root}.
Suppose $\tau_\Ic(\Ic^{3/2})=\{0\}$.
Assume either Hypotheses \ref{2cases} or take $H_0=H_0^*\in\mathcal{B}$, $V=V^*\in\Ic^{1/2}$ and
$\mathcal{F}=C^4(\Reals)$.
Then, for every $f\in\mathcal{F}$
\begin{equation*}
\tau_\Ic\left(f(H_0+V)-f(H_0)-\frac{d}{dt}\bigg|_{t=0}f(H_0+tV)\right)=
\frac12\,\tau_\Ic\left(\frac{d^2}{dt^2}\bigg|_{t=0}f(H_0+tV)\right).
\end{equation*}

\end{thm}

\begin{remark}
\label{2comparison}
It was proved in \cite{Koplienko} that in case of self-adjoint $H_0$, $V$, with $V$ in the Hilbert-Schmidt class and $\tau_\Ic$ replaced with the standard trace $\Tr$ in \eqref{2tf}, the second order spectral shift measure
can be expressed via the first order spectral shift measure, and is absolutely continuous. The latter proof crucially relied on the fact that a Hilbert-Schmidt operator can be approximated by a sequence of trace class operators in the Hilbert-Schmidt norm. In the case of $\Ic=\mathcal{L}^{(1,\infty)}$ and a Dixmier trace $\tau_\Ic=\Tr_\omega$, we do not have a similar approximation property for the elements of $\Ic^{1/2}$ by the elements of $\Ic$. Moreover, as a consequence of the singularity of the Dixmier trace, $\nu_{H_0,V}$ can be a singular measure (see Proposition \ref{2dm}) and if $V\in \Ic$, 
then $\nu_{H_0,V}$ degenerates to zero (see Proposition \ref{2triv}).
\end{remark}

\begin{prop}
\label{2dm} Assume Hypotheses \ref{ideal} and \ref{root}.
Suppose $\tau_\Ic(\Ic^{3/2})=\{0\}$.
Let $H_0=aI$, for $a\in\Reals$, and let $V=V^*\in\Ic^{1/2}$.
Then, under Hypotheses \ref{2cases}(i), the measure $\nu_{H_0,V}=\frac12\,\tau_\Ic(V^2)\,\delta_a$ on $\Reals$
satisfies Theorem~\ref{2mt}.
\end{prop}
\begin{proof}
By Theorem \ref{2mt} and direct calculations,
\begin{align*}
&k(k-1)\int_\Reals t^{k-2}\,\nu_{H_0,V}(dt)\\
&\quad=\tau_\Ic\left((aI+V)^k-a^kI-\frac{d}{dt}\bigg|_{t=0}(aI+tV)^k\right)
=\tau_\Ic\left(\sum_{j=1}^k\begin{pmatrix}k\\j\end{pmatrix}a^{k-j}V^j-ka^{k-1}V\right)\\
&\quad=k(k-1)\,a^{k-2}\,\frac12\,\tau_\Ic(V^2).
\end{align*}
The rest of the proof goes like the one of Proposition \ref{dm}.
\end{proof}

\begin{prop}
\label{2amc} Assume Hypotheses \ref{ideal} and \ref{root}.
Let $H_0$ and $H_0+V$ be contractions.
Assume that $\tau_\Ic(\Ic^{3/2})=\{0\}$.
If $H_0,V\in\Ic^{1/2}$, then under Hypotheses \ref{2cases}(ii),
Theorem~\ref{2mt} holds with an absolutely continuous measure $\nu_{H_0,V}$, and in fact with $\nu_{H_0,V}$ equal to
a multiple of Haar measure on the unit circle.
\end{prop}

\begin{proof}
The proof follows from adjusting the reasoning in the proof of Proposition \ref{amc}, where we employ Lemma \ref{2l1} \eqref{2l1ii} and Lemma \ref{ir} to show that $\int_\mathbb{T}z^n\,d\nu_{H_0,V}(z)=0$, for $n\in\Nats$.
\end{proof}

\begin{prop}
\label{2triv}
Assume Hypotheses \ref{ideal}, \ref{root}, and \ref{2cases}. If $\tau_\Ic(|V|^2)=0$, then Theorem~\ref{2mt} holds with
$\nu_{H_0,V}=0$.
\end{prop}

\begin{proof}
The result is an immediate consequence of the estimate \eqref{2tfe} for the total variation of the measure $\nu_{H_0,V}$.
\end{proof}

As in the case of a normal trace, the measure $\nu_{H_0,V}$ is nonnegative whenever $H_0$ and $V$ are bounded self-adjoint operators.

\begin{prop} Assume Hypotheses \ref{ideal} and \ref{root}.
Assume $H_0=H_0^*\in\mathcal{B}$ and $V=V^*\in\Ic^{1/2}$.
Then Theorem~\ref{2mt} (with Hypotheses~\ref{2cases}(i)) holds with the measure
$\nu_{H_0,V}$ on $\Reals$ nonnegative.
\end{prop}

\begin{proof}
In the course of the proofs of Theorems \ref{2mt} and \ref{2l4} (which we apply after rescaling the operators
to get contractions), we have established that for $f$ a polynomial,
\begin{align*}
&\tau_\Ic\left(f(H_0+V)-f(H_0)-\frac{d}{dt}\bigg|_{t=0}f(H_0+tV)\right)\\
&\quad=2\int_0^1 (1-t)\lim_{n\rightarrow\infty}\sum_{i_0=1}^n\sum_{i_1=1}^n f^{[2]}\big(\la_{i_0}^{(n)},\la_{i_1}^{(n)},\la_{i_0}^{(n)}\big)
\,\tau_\Ic\left(E_{H_t}\big(\delta_{i_0}^{(n)}\big)V E_{H_t}\big(\delta_{i_1}^{(n)}\big)V\right)\,dt.
\end{align*}
Denote $E_t(\la_k)=E_{H_0+tV}(d\la_k)$, for $k=0,1$.
Since
\begin{align*}
\left<E_t(\la_0)VE_t(\la_1)VE_t(\la_0)h,h\right>&=\left<E_t(\la_1)VE_t(\la_0)h,VE_t(\la_0)h\right>\geq 0,
\end{align*}
for any $h\in\Hcal$, we see that the set functions
$\tau_\Ic\big(E_t(\la_0)V E_t(\la_1)V\big)
$ are nonnegative. Therefore, $\tau_\Ic\left(f(H_0+V)-f(H_0)-\frac{d}{dt}\big|_{t=0}f(H_0+tV)\right)\geq 0$ whenever $f''\geq 0$ (on a segment containing the spectra of operators $H_0+tV$, $t\in[0,1]$). Finally, application of \eqref{2tf} completes the proof.
\end{proof}

\end{document}